\documentclass[oneside,reqno,english]{amsart}
\usepackage{lmodern}
\usepackage[T1]{fontenc}
\usepackage[latin9]{inputenc}
\usepackage{babel}
\usepackage{varioref}
\usepackage{float}
\usepackage{amstext}
\usepackage{amsthm}
\usepackage{amssymb}
\PassOptionsToPackage{normalem}{ulem}
\usepackage{ulem}
\usepackage[unicode=true,pdfusetitle,
 bookmarks=true,bookmarksnumbered=false,bookmarksopen=false,
 breaklinks=false,pdfborder={0 0 1},backref=false,colorlinks=false]
 {hyperref}
\usepackage{breakurl}

\makeatletter

\floatstyle{ruled}
\newfloat{algorithm}{tbp}{loa}
\providecommand{\algorithmname}{Algorithm}
\floatname{algorithm}{\protect\algorithmname}

\numberwithin{equation}{section}
\numberwithin{figure}{section}
\theoremstyle{plain}
\newtheorem{thm}{\protect\theoremname}[section]
\theoremstyle{plain}
\newtheorem{lyxalgorithm}[thm]{\protect\algorithmname}
\theoremstyle{plain}
\newtheorem{lem}[thm]{\protect\lemmaname}
\theoremstyle{remark}
\newtheorem{rem}[thm]{\protect\remarkname}
\theoremstyle{plain}
\newtheorem{prop}[thm]{\protect\propositionname}
\theoremstyle{definition}
\newtheorem{defn}[thm]{\protect\definitionname}
\theoremstyle{definition}
\newtheorem{example}[thm]{\protect\examplename}
\theoremstyle{plain}
\newtheorem{assumption}[thm]{\protect\assumptionname}
\theoremstyle{remark}
\newtheorem{claim}[thm]{\protect\claimname}

\newcommand{\dom}{\mbox{\rm dom}}

\makeatother

\providecommand{\algorithmname}{Algorithm}
\providecommand{\assumptionname}{Assumption}
\providecommand{\claimname}{Claim}
\providecommand{\definitionname}{Definition}
\providecommand{\examplename}{Example}
\providecommand{\lemmaname}{Lemma}
\providecommand{\propositionname}{Proposition}
\providecommand{\remarkname}{Remark}
\providecommand{\theoremname}{Theorem}

\begin{document}
\title[Subdifferentiable Dykstra splitting]{Subdifferentiable functions and partial data communication in a distributed deterministic asynchronous Dykstra's algorithm}

\begin{abstract}
We described a decentralized distributed deterministic asynchronous
Dykstra's algorithm that allows for time-varying graphs in an earlier
paper. In this paper, we show how to incorporate subdifferentiable
functions into the framework using a step similar to the bundle method.
We point out that our algorithm also allows for partial data communications.
We discuss a standard step for treating the composition of a convex
and linear function.
\end{abstract}

\author{C.H. Jeffrey Pang}

\thanks{C.H.J. Pang acknowledges grant R-146-000-214-112 from the Faculty
of Science, National University of Singapore. }

\curraddr{Department of Mathematics\\ 
National University of Singapore\\ 
Block S17 08-11\\ 
10 Lower Kent Ridge Road\\ 
Singapore 119076 }

\email{matpchj@nus.edu.sg}

\date{\today{}}


\maketitle
\tableofcontents{}

\section{Introduction}

Consider a connected graph $\mathcal{G}=(\mathcal{V},\mathcal{E})$
where a closed convex function $f_{i}:\mathbb{R}^{m}\to\mathbb{R}\cup\{\infty\}$
is defined on each vertex $i\in\mathcal{V}$. A problem of interest
that occurs in problems with data too large to be stored in a single
location is to minimize the sum 
\begin{equation}
\begin{array}{c}
\underset{x\in\mathbb{R}^{m}}{\min}\underset{i\in\mathcal{V}}{\sum}f_{i}(x)\end{array}\label{eq:distrib-primal}
\end{equation}
in a distributed manner so that the communications of data occur only
along the edges of the graph. In our earlier paper \cite{Pang_Dist_Dyk},
we consider the regularized problem 
\begin{equation}
\begin{array}{c}
\underset{x\in\mathbb{R}^{m}}{\min}\underset{i\in\mathcal{V}}{\sum}[f_{i}(x)+\frac{1}{2}\|x-[\bar{\mathbf{{x}}}]_{i}\|^{2}]\end{array}\label{eq:distrib-dyk-primal}
\end{equation}
instead, where $\bar{\mathbf{{x}}}\in[\mathbb{R}^{m}]^{|\mathcal{V}|}$. 

\subsection{\label{subsec:Distrib-algs}Distributed optimization algorithms}

Since this paper builds on \cite{Pang_Dist_Dyk}, we shall give a
brief introduction. Our algorithm is for the case when the edges are
undirected. But we remark that notable papers on the directed case.
A notable paper based on the directed case using the subgradient algorithm
is \cite{EXTRA_Shi_Ling_Wu_Yin}, and surveys are \cite{Nedich_survey}
and \cite{Nedich_talk_2017}. The papers \cite{Nedich_Olshevsky}
and \cite{Nedich_Olshevsky_Shi} further touch on the case of time-varying
graphs. The algorithm in \cite{Notarstefano_gang_Newton_2017} uses
a Newton-Raphson method to design a distributed algorithm for directed
graphs. Naturally, the communication requirements for directed graphs
need to be more stringent that the requirements for undirected graphs. 

From here on, we discuss only algorithms for undirected graphs. A
product space formulation on the ADMM leads to a distributed algorithm
\cite[Chapter 7]{Boyd_Eckstein_ADMM_review}. Such an algorithm is
decentralized and distributed, but is not asynchronous and so can
get slowed down by slow vertices. An approach based on \cite{Eckstein_Combettes_MAPR}
allows for asynchronous operation, but is not decentralized. 

Moving beyond deterministic algorithms, distributed decentralized
asynchronous algorithms were proposed, but many of them involve some
sort of randomization. For example, the work \cite{Iutzeler_Bianchi_Ciblat_Hachem_1st_paper_dist,Bianchi_Hachem_Iutzeler_2nd_paper_dist}
and the generalization \cite{AROCK_Peng_Xu_Yan_Yin} are based on
monotone operator theory (see for example the textbook \cite{BauschkeCombettes11}),
and require the computations in the nodes to follow specific probability
distributions. 

We now look at asynchronous distributed algorithm with deterministic
convergence (rather than probabilistic convergence). Other than subgradient
methods, we mention that the paper \cite{Aytekin_F_Johansson_2016}
is an algorithm for strongly convex problems that is primal in nature,
so can't handle constraint sets as is. The method in \cite{Aybat_Hamedani_2016}
may arguably be considered to have these properties.

\subsection{\label{subsec:Dyk-method}Dykstra's algorithm and the corresponding
distributed algorithm}

Again, we shall be brief with the introduction, and defer to \cite{Pang_Dist_Dyk}
for a more detailed introduction. Dykstra's algorithm was first studied
in \cite{Dykstra83} for projecting a point onto the intersection
of a number of closed convex sets. The convergence proof without the
existence of dual solutions was established in \cite{BD86} and rewritten
in terms of duality in \cite{Gaffke_Mathar}, and is sometimes called
the Boyle-Dykstra theorem. Dykstra's algorithm was independently noted
in \cite{Han88} to be block coordinate minimization on the dual problem,
but their proof depends on the existence of a dual solution. (For
an example of a problem without dual solutions, look at \cite[page 9]{Han88}
where two circles in $\mathbb{R}^{2}$ intersect at only one point.)
We pointed out in \cite{Pang_Dyk_spl} that the Boyle-Dykstra theorem
can be extended to the case of minimization problems of the form $\min_{x}\frac{1}{2}\|x-\bar{x}\|^{2}+\sum_{i=1}^{k}f_{i}(x)$.
For more on the background on Dykstra's algorithm, we refer to \cite{BauschkeCombettes11,BB96_survey,Deustch01,EsRa11}.

Dykstra's algorithm was extended to a distributed algorithm in \cite{Borkar_distrib_dyk},
and they highlight the works \cite{Aybat_Hamedani_2016,LeeNedich2013,Nedic_Ram_Veeravalli_2010,Ozdaglar_Nedich_Parrilo}
on distributed optimization. The work in \cite{Borkar_distrib_dyk}
is vastly different from how Dykstra's algorithm is studied in \cite{BD86}
and \cite{Gaffke_Mathar}.

It turns out that \cite{Notars_asyn_distrib_2015} also makes use
of the same Dykstra's algorithm setting, but they solve with a randomized
dual proximal gradient method. The differences between their setup
and ours is detailed in \cite{Pang_Dist_Dyk}.

In \cite{Pang_Dist_Dyk}, we rewrote \eqref{eq:distrib-dyk-primal}
in a form similar to \eqref{eq:Dyk-primal} (see Remark \ref{rem:partial-comms-change}
for an explanation of the differences) and applied an extended Dykstra's
algorithm. We list down the features of the distributed Dyksyra's
algorithm: 
\begin{enumerate}
\item distributed (with communications occurring only between adjacent agents
$i$ and $j$ connected by an edge), 
\item decentralized (i.e., there is no central node coordinating calculations), 
\item asynchronous (contrast this to synchronous algorithms, where the faster
agents would wait for slower agents before they can perform their
next calculations), 
\item able to allow for time-varying graphs in the sense of \cite{Nedich_Olshevsky,Nedich_Olshevsky_Shi}
(to be robust against failures of communication between two agents),
\item deterministic (i.e., not using any probabilistic methods, like stochastic
gradient methods),
\item able to allow for constrained optimization, where the feasible region
is the intersection of several sets (this largely rules out primal-only
methods), 
\item able to incorporate proximable functions naturally. 
\end{enumerate}
Since Dykstra's algorithm is also dual block coordinate ascent, the
following property is obtained:
\begin{enumerate}
\item [(8)]choosing large number of dual variables to be maximized over
gives a greedier increase of the dual objective value. 
\end{enumerate}
Also, the distributed Dykstra's algorithm does not require the existence
of a dual minimizer provided that the functions $f_{i}(\cdot)$ are
proximable. Moreover, if some of the $f_{i}(\cdot)$ were defined
to be the indicator functions of closed convex sets, then a greedy
step for dual ascent \cite{Pang_DBAP} is possible. For the rest of
this paper, we shall just refer to the algorithm in \cite{Pang_Dist_Dyk}
as the distributed Dykstra's algorithm.

\subsection{Main contribution of this paper}

This paper builds on \cite{Pang_Dist_Dyk}. We now describe the main
contribution of this paper without assuming any prior knowledge of
\cite{Pang_Dist_Dyk}. 

For each node $i\in\mathcal{V}$, recall the function $f_{i}:\mathbb{R}^{m}\to\mathbb{R}$
in \eqref{eq:distrib-dyk-primal}. Let $\mathbf{f}_{i}:[\mathbb{R}^{m}]^{|\mathcal{V}|}\to\bar{\mathbb{R}}$
be defined by $\mathbf{f}_{i}(\mathbf{{x}})=f_{i}(x_{i})$ (i.e.,
$\mathbf{f}_{i}$ depends only on $i$-th variable, where $i\in\mathcal{V}$).
Recall the graph $\mathcal{G}=(\mathcal{V},\mathcal{E})$. Let the
set $\bar{\mathcal{E}}$ be defined to be 
\[
\bar{\mathcal{E}}:=\mathcal{E}\times\{1,\dots,m\}.
\]
For $\mathbf{{x}}\in[\mathbb{R}^{m}]^{|\mathcal{V}|}$, the component
$\mathbf{{x}}_{i}\in\mathbb{R}^{m}$ is straightforward. We let $[\mathbf{{x}}_{i}]_{\bar{k}}\in\mathbb{R}$
be the $\bar{k}$-th component of $\mathbf{{x}}_{i}\in\mathbb{R}^{m}$.
For each $((i,j),\bar{k})\in\bar{\mathcal{E}}$, the hyperplane $H_{((i,j),\bar{k})}\subset[\mathbb{R}^{m}]^{|\mathcal{V}|}$
is defined to be 
\begin{equation}
H_{((i,j),\bar{k})}:=\{\mathbf{{x}}\in[\mathbb{R}^{m}]^{|\mathcal{V}|}:[\mathbf{{x}}_{i}]_{\bar{k}}=[\mathbf{{x}}_{j}]_{\bar{k}}\}.\label{eq:def-halfspaces}
\end{equation}

We can see that the regularized problem \eqref{eq:distrib-dyk-primal}
is equivalent to 
\begin{equation}
\min_{\mathbf{{x}}\in[\mathbb{R}^{m}]^{|\mathcal{V}|}}\frac{1}{2}\|\mathbf{{x}}-\bar{\mathbf{{x}}}\|^{2}+\sum_{((i,j),\bar{k})\in\bar{\mathcal{E}}}\underbrace{\delta_{H_{((i,j),\bar{k})}}(\mathbf{{x}})}_{\mathbf{f}_{((i,j),\bar{k})}(\mathbf{{x}})}+\sum_{i\in\mathcal{V}}\mathbf{{f}}_{i}(\mathbf{{x}}).\label{eq:Dyk-primal}
\end{equation}
We let the functions $\mathbf{f}_{\alpha}:[\mathbb{R}^{m}]^{|\mathcal{V}|}\to\bar{\mathbb{R}}$
be as defined in \eqref{eq:Dyk-primal} for all $\alpha\in\bar{\mathcal{E}}\cup\mathcal{V}$.
The (Fenchel) dual of \eqref{eq:Dyk-primal} is

\begin{equation}
\max_{\mathbf{{z}}_{\alpha}\in[\mathbb{R}^{m}]^{|\mathcal{V}|},\alpha\in\bar{\mathcal{E}}\cup\mathcal{V}}F(\{\mathbf{{z}}_{\alpha}\}_{\alpha\in\bar{\mathcal{E}}\cup\mathcal{V}}),\label{eq:dual-fn}
\end{equation}
where 
\begin{equation}
\begin{array}{c}
F(\{\mathbf{{z}}_{\alpha}\}_{\alpha\in\bar{\mathcal{E}}\cup\mathcal{V}}):=-\frac{1}{2}\bigg\|\bar{\mathbf{{x}}}-\underset{\alpha\in\bar{\mathcal{E}}\cup\mathcal{V}}{\sum}\mathbf{{z}}_{\alpha}\bigg\|^{2}+\frac{1}{2}\|\bar{\mathbf{{x}}}\|^{2}-\underset{\alpha\in\bar{\mathcal{E}}\cup\mathcal{V}}{\sum}\mathbf{f}_{\alpha}^{*}(\mathbf{{z}}_{\alpha}).\end{array}\label{eq:Dykstra-dual-defn}
\end{equation}
To give further insight on the problems \eqref{eq:Dyk-primal}-\eqref{eq:Dykstra-dual-defn},
we note that if $\mathbf{{f}}_{i}\equiv0$, then the problems \eqref{eq:Dyk-primal}
reduces to the averaged consensus algorithm in \cite{Boyd_distrib_averaging,Distrib_averaging_Dimakis_Kar_Moura_Rabbat_Scaglione},
where the primal variable $\mathbf{{x}}\in[\mathbb{R}^{m}]^{|\mathcal{V}|}$
converges to the vector $\mathbf{{x}}^{*}\in[\mathbb{R}^{m}]^{|\mathcal{V}|}$
(where $\mathbf{{x}}^{*}$ is defined so that each $\mathbf{{x}}_{i}^{*}\in\mathbb{R}^{m}$
is the average $\frac{1}{|\mathcal{V}|}\sum_{i\in\mathcal{V}}\bar{\mathbf{{x}}}_{i}$)
at a linear rate dependent on the properties of the graph $(\mathcal{V},\mathcal{E})$. 

In \cite{Pang_Dist_Dyk}, we applied the techniques of \cite{Gaffke_Mathar,Hundal-Deutsch-97}
to prove that a block coordinate optimization applied to \eqref{eq:Dykstra-dual-defn}
leads to the increase of the objective value $F(\cdot)$ in \eqref{eq:Dykstra-dual-defn}
to the maximal value, which is also the objective value of \eqref{eq:Dyk-primal}
since strong duality can also be proven. Further work in \cite{Pang_Dist_Dyk}
shows that the algorithm has properties (1)-(5). 

We now note that block coodinate optimization applied to \eqref{eq:Dykstra-dual-defn}
can be easily carried out only if the functions $\mathbf{f}_{i}(\cdot)$
are proximable. For illustration, we suppose that we only minimize
with respect to $\mathbf{{z}}_{i^{*}}$ for some $i^{*}\in\mathcal{V}$,
but leave all other $\{\mathbf{{z}}_{\alpha}\}_{\alpha\in[\mathcal{V}\cup\mathcal{E}]\backslash\{i^{*}\}}$
fixed. We showed in \cite{Pang_Dist_Dyk} that $\mathbf{{z}}_{i^{*}}$
is sparse, with $[\mathbf{{z}}_{i^{*}}]_{j}=0$ whenever $i^{*}\neq j$
(see Proposition \ref{prop:sparsity}), so employing techniques in
\cite{Pang_Dist_Dyk} (see Proposition \ref{prop:subproblems}) shows
that the primal problem to be solved is 
\[
\mathbf{{x}}_{i^{*}}=\min_{\mathbf{{x}}'_{i^{*}}\in\mathbb{R}^{m}}\frac{1}{2}\|\mathbf{{x}}'_{i^{*}}-[\bar{\mathbf{{x}}}-\sum_{\alpha\in[\mathcal{E}\cup\mathcal{V}]\backslash\{i^{*}\}}\mathbf{{z}}_{i}]_{i^{*}}\|^{2}+f_{i}(\mathbf{{x}}'_{i^{*}}),
\]
and the $[z_{i^{*}}]_{i^{*}}$ is the corresponding dual variable.
This means that $f_{i}(\cdot)$ has to be proximable. 

As it stands, the algorithm in \cite{Pang_Dist_Dyk} does not handle
the case when $f_{i}(\cdot)$ are smooth for all $i\in\mathcal{V}$.
Given an affine minorant of $f_{i^{*}}(\cdot)$, say $\tilde{f}_{i^{*}}(\cdot)$,
the conjugate $\tilde{f}_{i^{*}}^{*}(\cdot)$ satisfies $\tilde{f}_{i^{*}}^{*}(\cdot)\geq f_{i^{*}}^{*}(\cdot)$.
The main contribution of this paper is to show that for the dual function
\eqref{eq:Dykstra-dual-defn}, if the $f_{i}^{*}(\cdot)$ are replaced
by $\tilde{f}_{i}^{*}(\cdot)$ whenever $f_{i}(\cdot)$ is a subdifferentiable
function and $\tilde{f}_{i}(\cdot)$ is defined as an affine minorant
of $f_{i}(\cdot)$, then the minorized dual functions would have the
values ascending and converging to the optimal objective value of
the dual problem \eqref{eq:dual-fn}. This extends the algorithm in
\cite{Pang_Dist_Dyk} to give an algorithm with properties (1)-(8)
and also incorporating subdifferentiable $f_{i}(\cdot)$ naturally.
(A more traditional method of majorizing $f_{i^{*}}^{*}(\cdot)$ through
$f_{i^{*}}^{*}([\mathbf{z}_{i^{*}}]_{i^{*}})+\langle\mathbf{x}_{i},z-[\mathbf{z}_{i^{*}}]_{i^{*}}\rangle+\frac{\sigma}{2}\|z-[\mathbf{z}_{i^{*}}]_{i^{*}}\|^{2}$
would be problematic because a strongly convex modulus $\sigma$ of
$f_{i^{*}}^{*}(\cdot)$ may not even exist, which is the case when
$f_{i^{*}}(\cdot)$ is affine.) As far as we are aware, distributed
algorithms for subdifferentiable functions include methods based on
the subgradient algorithm as mentioned earlier as well as \cite{Wang_Bertsekas_incremental_unpub_2013}.
(Since the problems we treat in this paper are strongly convex, it
would be unfair to bring out the fact that subgradient methods are
slow for problems that are not strongly convex due to the need of
using diminishing stepsizes. But still, our dual approach has other
advantages compared to the subgradient algorithm since not all of
properties (1)-(8) are satisfied by the subgradient algorithm.)

In Section \ref{sec:First-alg}, we first show that this procedure
is sound for the sum of one subdifferentiable function and a regularizing
quadratic with convergence rates compatible with standard first order
methods. In Section \ref{sec:main-alg}, we integrate this algorithm
into our distributed Dykstra's algorithm.

\subsection{Other contributions of this paper}

 In \cite{Pang_Dist_Dyk}, we had used the hyperplanes $H_{(i,j)}:=\{\mathbf{x}\in\mathbf{X}:[\mathbf{x}]_{i}=[\mathbf{x}]_{j}\}$
for all $(i,j)\in\mathcal{E}$ instead of \eqref{eq:def-halfspaces}.
We point out that using $H_{((i,j),k)}$ instead of $H_{(i,j)}$ allows
for part of the data to be communicated at one time step to achieve
convergence to the optimal solution, which in turn means that computation
will not be held back by communications between nodes. See Subsection
\ref{subsec:Partial-comm-prelim} and Example \ref{exa:partial-comms}
for more details. 

Finally, in Subsection \ref{subsec:composition-lin-op}, we point
out that a standard step allows us to reduce matrix operations whenever
the function $f_{i}(\cdot)$ of the form $\tilde{f}_{i}\circ A_{i}$
for some closed convex function $\tilde{f}_{i}(\cdot)$ and linear
map $A_{i}$, although such a step now introduces additional regularizing
functions.

\subsection{Notation}

For much of the paper, we will be looking at functions with domain
either $\mathbb{R}^{m}$ or $[\mathbb{R}^{m}]^{|\mathcal{V}|}$. We
reserve bold letters for functions with domain $[\mathbb{R}^{m}]^{|\mathcal{V}|}$
(for example, \eqref{eq:Dyk-primal}), and we usually use non-bold
letters for functions with domain $\mathbb{R}^{m}$ (for example,
\eqref{eq:small-pblm}). For a vector $\mathbf{{z}}\in[\mathbb{R}^{m}]^{|\mathcal{V}|}$,
$\mathbf{{z}}_{i}\in\mathbb{R}^{m}$ and $[\mathbf{{z}}]_{i}\in\mathbb{R}^{m}$
are both understood to be the $i$-th component of $\mathbf{{z}}$,
where $i\in|\mathcal{V}|$. Furthermore, $[\mathbf{{z}}_{i}]_{\bar{k}}$
and $[[\mathbf{{z}}]_{i}]_{\bar{k}}\in\mathbb{R}$ are both understood
to be the $\bar{k}$-th component of $[\mathbf{{z}}]_{i}$. 

We say that $f(\cdot)$ is proximable if the problem $\arg\min_{x}f(x)+\frac{1}{2}\|x-\bar{x}\|^{2}$
is easy to solve for any $\bar{x}$. For a closed convex set $C$,
the indicator function is denoted by $\delta_{C}(\cdot)$. All other
notation are standard.

\section{\label{sec:First-alg}The algorithm for one function}

In this section, we consider the problem 
\begin{equation}
\begin{array}{c}
\underset{x\in\mathbb{R}^{m}}{\min}f(x)+\frac{1}{2}\|x-\bar{x}\|^{2},\end{array}\label{eq:small-pblm}
\end{equation}
where $f:\mathbb{R}^{m}\to\mathbb{R}$ is a subdifferentiable convex
function such that $\dom(f)=\mathbb{R}^{m}$. We define our first
dual ascent algorithm to solve \eqref{eq:small-pblm} before we show
how to integrate it into the distributed Dykstra's algorithm for solving
\eqref{eq:Dyk-primal} through the increasing the dual objective value
in \eqref{eq:dual-fn}-\eqref{eq:Dykstra-dual-defn}. Consider Algorithm
\vref{alg:basic-dual-ascent}, which is somewhat like the bundle method.

\begin{algorithm}[!h]
\begin{lyxalgorithm}
\label{alg:basic-dual-ascent}In this algorithm, we want to solve
\eqref{eq:small-pblm} 

Let $h_{0}:\mathbb{R}^{m}\to\mathbb{R}$ be an affine function such
that $h_{0}(\cdot)\leq f(\cdot)$ defined by the parameters $(\tilde{x}_{0},\tilde{f}_{0},\tilde{y}_{0})\in\mathbb{R}^{m}\times\mathbb{R}\times\mathbb{R}^{m}$
where for all $w\geq0$, $h_{w}:\mathbb{R}^{m}\to\mathbb{R}$ is defined
through $(\tilde{x}_{w},\tilde{f}_{w},\tilde{y}_{w})\in\mathbb{R}^{m}\times\mathbb{R}\times\mathbb{R}^{m}$
by 
\begin{equation}
h_{w}(x)=\tilde{y}_{w}^{T}(x-\tilde{x}_{w})+\tilde{f}_{w}.\label{eq:def-q}
\end{equation}

Without loss of generality, let $\tilde{x}_{0}$ be the minimizer
to $\min_{x}h_{0}(x)+\frac{1}{2}\|x-\bar{x}\|^{2}$.

01 For $w=0,\dots$ 

02 $\quad$Recall $\tilde{x}_{w}$ is the minimizer to $\min_{x}h_{w}(x)+\frac{1}{2}\|x-\bar{x}\|^{2}$.

03 $\quad$Evaluate $f(\tilde{x}_{w})$ and find a subgradient $\tilde{s}_{w}\in\partial f(\tilde{x}_{w})$. 

04 $\quad$Construct the affine function $\tilde{h}_{w}:\mathbb{R}^{m}\to\mathbb{R}$
to be 
\begin{equation}
\tilde{h}_{w}(x)=\tilde{s}_{w}^{T}(x-\tilde{x}_{w})+f(\tilde{x}_{w}).\label{eq:def-h-tilde-w}
\end{equation}

05 $\quad$Consider 

\begin{equation}
\begin{array}{c}
\underset{x}{\min}[\max\{\tilde{h}_{w},h_{w}\}(x)+\frac{1}{2}\|x-\bar{x}\|^{2}].\end{array}\label{eq:max-quad}
\end{equation}

06 $\quad$Let $\tilde{x}_{w+1}$ be the minimizer of \eqref{eq:max-quad}.

07 $\quad$Let $\tilde{f}_{w+1}=\max\{\tilde{h}_{w},h_{w}\}(\tilde{x}_{w+1})$.

08 $\quad$Let $\tilde{y}_{w+1}$ be $\bar{x}-\tilde{x}_{w+1}$. 

09 $\quad$Define $h_{w+1}(\cdot)$ through $(\tilde{x}_{w+1},\tilde{f}_{w+1},\tilde{y}_{w+1})$
and \eqref{eq:def-q}. 

10 End for 
\end{lyxalgorithm}

\end{algorithm}

We shall prove that each function of the form \eqref{eq:def-q} is
a lower approximation of $f(\cdot)$ in Lemma \ref{lem:h-w-leq-f}.
With a sequence of such lower approximations like in the bundle method,
we can then solve \eqref{eq:small-pblm}. We prove some lemmas before
proving the convergence of Algorithm \ref{alg:basic-dual-ascent}. 
\begin{lem}
\label{lem:h-w-leq-f}In Algorithm \ref{alg:basic-dual-ascent}, the
functions $h_{w}(\cdot)$ are such that $h_{w}(\cdot)\leq f(\cdot)$.
\end{lem}

\begin{proof}
We prove our result by induction. Note that $h_{0}(\cdot)$ was defined
so that $h_{0}(\cdot)\leq f(\cdot)$. It is also clear from the definition
of $\tilde{h}_{w}(\cdot)$ in \eqref{eq:def-h-tilde-w} that $\tilde{h}_{w}(\cdot)\leq f(\cdot)$
for all $w\geq0$. Suppose that $h_{w}(\cdot)\leq f(\cdot)$. We now
show that $h_{w+1}(\cdot)\leq f(\cdot)$. We have 
\begin{equation}
\max\{\tilde{h}_{w},h_{w}\}(\cdot)\leq f(\cdot).\label{eq:max-leq-f}
\end{equation}
The functions $\tilde{h}_{w}(\cdot)$ and $h_{w}(\cdot)$ are convex,
so $\max\{\tilde{h}_{w},h_{w}\}(\cdot)$ is convex. Since the minimum
of $\max\{\tilde{h}_{w},h_{w}\}(\cdot)+\frac{1}{2}\|\cdot-\bar{x}\|^{2}$
is attained at $\tilde{x}_{w+1}$, it follows that $0\in\partial[\max\{\tilde{h}_{w},h_{w}\}](\tilde{x}_{w+1})+\tilde{x}_{w+1}-\bar{x}$,
or that $\bar{x}-\tilde{x}_{w+1}\in\partial[\max\{\tilde{h}_{w},h_{w}\}](\tilde{x}_{w+1})$.
The construction of $h_{w+1}(\cdot)$ implies that $h_{w+1}(\tilde{x}_{w+1})=\max\{\tilde{h}_{w},h_{w}\}(\tilde{x}_{w+1})$
and $h_{w+1}(\cdot)\leq\max\{\tilde{h}_{w},h_{w}\}(\cdot)$. Together
with \eqref{eq:max-leq-f}, this implies $h_{w+1}(\cdot)\leq f(\cdot)$,
which completes the proof. 
\end{proof}
Let 
\begin{equation}
\bar{\alpha}_{w}:=f(\tilde{x}_{w})-h_{w}(\tilde{x}_{w}).\label{eq:bar-alpha-w}
\end{equation}
Let the minimizer of \eqref{eq:small-pblm} be $x^{*}$. We have 
\begin{eqnarray}
\begin{array}{c}
h_{w}(\tilde{x}_{w})+\frac{1}{2}\|\tilde{x}_{w}-\bar{x}\|^{2}\end{array} & \overset{\scriptsize{\mbox{Lem }\ref{lem:h-w-leq-f}\mbox{, Alg \ref{alg:basic-dual-ascent} line 2}}}{\leq} & \begin{array}{c}
f(x^{*})+\frac{1}{2}\|x^{*}-\bar{x}\|^{2}\end{array}\label{eq:value-chain}\\
 & \overset{\scriptsize{x^{*}\mbox{ solves \eqref{eq:small-pblm}}}}{\leq} & \begin{array}{c}
f(\tilde{x}_{w})+\frac{1}{2}\|\tilde{x}_{w}-\bar{x}\|^{2}.\end{array}\nonumber 
\end{eqnarray}
Let the real number $\alpha_{w}$ be 
\begin{equation}
\begin{array}{c}
\alpha_{w}:=\left[f(x^{*})+\frac{1}{2}\|x^{*}-\bar{x}\|^{2}\right]-\left[h_{w}(\tilde{x}_{w})+\frac{1}{2}\|\tilde{x}_{w}-\bar{x}\|^{2}\right].\end{array}\label{eq:alpha-w}
\end{equation}
It is clear to see that \eqref{eq:value-chain} translates to $0\leq\alpha_{w}\leq\bar{\alpha}_{w}$. 
\begin{lem}
\label{lem:alpha-recurrs}Recall the definitions of $\alpha_{w}$
and $\bar{\alpha}_{w}$ in \eqref{eq:bar-alpha-w} and \eqref{eq:alpha-w}.

\begin{enumerate}
\item We have $\alpha_{w+1}\leq\alpha_{w}-\frac{1}{2}t^{2}$, where $\frac{1}{2}t^{2}+t\|\tilde{s}_{w}+\tilde{x}_{w}-\bar{x}\|=\bar{\alpha}_{w}$.
\item Next, 
\begin{equation}
\begin{array}{c}
\frac{1}{2\|\tilde{s}_{w}+\tilde{x}_{w}-\bar{x}\|^{2}}\alpha_{w+1}^{2}+\alpha_{w+1}\leq\alpha_{w}.\end{array}\label{eq:target-quad}
\end{equation}
\end{enumerate}
\end{lem}

\begin{proof}
Since the function $x\mapsto\tilde{h}_{w}(x)+\frac{1}{2}\|x-\bar{x}\|^{2}$
is convex, it is bounded from below by its linearization at $\tilde{x}_{w}$
using $\tilde{s}_{w}\in\partial f(\tilde{x}_{w})$ via \eqref{eq:def-h-tilde-w}.
In other words, for all $x\in\mathbb{R}^{m}$, we have 
\begin{equation}
\begin{array}{c}
\tilde{h}_{w}(x)+\frac{1}{2}\|x-\bar{x}\|^{2}\overset{\eqref{eq:def-h-tilde-w}}{\geq}\underbrace{(\tilde{s}_{w}+\tilde{x}_{w}-\bar{x})^{T}(x-\tilde{x}_{w})+f(\tilde{x}_{w})+\frac{1}{2}\|\tilde{x}_{w}-\bar{x}\|^{2}}_{l_{(\tilde{x}_{w},\tilde{s}_{w})}(x)},\end{array}\label{eq:q-hat-lower-bdd}
\end{equation}
where $l_{(\tilde{x}_{w},\tilde{s}_{w})}(\cdot)$ as defined above
is the affine function derived from taking a subgradient of $\tilde{h}_{w}(\cdot)+\frac{1}{2}\|\cdot-\bar{x}\|^{2}$
at $\tilde{x}_{w}$. Consider the problem 
\begin{equation}
\begin{array}{c}
\underset{x}{\min}\underbrace{\max\left\{ h_{w}(x)+\frac{1}{2}\|x-\bar{x}\|^{2},l_{(\tilde{x}_{w},\tilde{s}_{w})}(x)\right\} }_{h_{w}^{+}(x)},\end{array}\label{eq:def-h-plus-w}
\end{equation}
where $h_{w}^{+}(\cdot)$ is as defined above. 

We first look at the case when $\tilde{s}_{w}+\tilde{x}_{w}-\bar{x}=0$.
We have $h_{w}(\tilde{x}_{w})\leq f(\tilde{x}_{w})$ by Lemma \ref{lem:h-w-leq-f}.
So 
\begin{align*}
 & f(\tilde{x}_{w})+\frac{1}{2}\|\tilde{x}_{w}-\bar{x}\|^{2}\overset{\tilde{s}_{w}+\tilde{x}_{w}-\bar{x}=0}{=}\min_{x}l_{(\tilde{x}_{w},\tilde{s}_{w})}(x)\overset{\eqref{eq:def-h-plus-w}}{\leq}\min_{x}h_{w}^{+}(x)\\
\leq & h^{+}(\tilde{x}_{w})\overset{\scriptsize{\text{Lemma \ref{lem:h-w-leq-f}}}}{\leq}f(\tilde{x}_{w})+\frac{1}{2}\|\tilde{x}_{w}-\bar{x}\|^{2}.
\end{align*}
We then have $h_{w}(\tilde{x}_{w})=f(\tilde{x}_{w})$, or $\alpha_{w}=0$.
Also, 
\[
\begin{array}{c}
0=\tilde{s}_{w}+\tilde{x}_{w}-\bar{x}\in\partial f(\tilde{x}_{w})+\partial(\frac{1}{2}\|\cdot-\bar{x}\|^{2})(\tilde{x}_{w}),\end{array}
\]
so $\tilde{x}_{w}=x^{*}$. The remaining techniques in this proof
shows that $\alpha_{w'}=0$ for all $w'\geq w$, which implies the
claims in this lemma. Thus, we assume $\tilde{s}_{w}+\tilde{x}_{w}-\bar{x}\neq0$.

Since $l_{(\tilde{x}_{w},\tilde{s}_{w})}(\cdot)$ is affine and $h_{w}(\cdot)+\frac{1}{2}\|\cdot-\bar{x}\|^{2}$
is a quadratic with minimizer $\tilde{x}_{w}$ and Hessian $I$, some
elementary calculations will show that the minimizer of \eqref{eq:def-h-plus-w}
is of the form $\tilde{x}_{w}-t\frac{\tilde{s}_{w}+\tilde{x}_{w}-\bar{x}}{\|\tilde{s}_{w}+\tilde{x}_{w}-\bar{x}\|}$
for some $t\geq0$. Let $\tilde{d}_{w}:=\frac{\tilde{s}_{w}+\tilde{x}_{w}-\bar{x}}{\|\tilde{s}_{w}+\tilde{x}_{w}-\bar{x}\|}$,
and let this minimizer be $\tilde{x}_{w}-t\tilde{d}_{w}$. We can
see that $t>0$, because if $t=0$, then $\tilde{x}_{w}-t\tilde{d}_{w}=\tilde{x}_{w}$,
and $\tilde{x}_{w}$ would once again be the minimizer of $f(\cdot)+\frac{1}{2}\|\cdot-\bar{x}\|^{2}$.
With $t>0$, the function values in \eqref{eq:def-h-plus-w} are equal,
which gives
\begin{equation}
\begin{array}{c}
h_{w}(\tilde{x}_{w}-t\tilde{d}_{w})+\frac{1}{2}\|(\tilde{x}_{w}-t\tilde{d}_{w})-\bar{x}\|^{2}\overset{\scriptsize{\mbox{Alg \ref{alg:basic-dual-ascent}, line 2}}}{=}h_{w}(\tilde{x}_{w})+\frac{1}{2}\|\tilde{x}_{w}-\bar{x}\|^{2}+\frac{1}{2}t^{2},\end{array}\label{eq:t-square-term}
\end{equation}
and 
\[
\begin{array}{c}
l_{(\tilde{x}_{w},\tilde{s})}(\tilde{x}_{w}-t\tilde{d}_{w})\overset{\eqref{eq:bar-alpha-w},\eqref{eq:q-hat-lower-bdd}}{=}h_{w}(\tilde{x}_{w})+\frac{1}{2}\|\tilde{x}_{w}-\bar{x}\|^{2}+\bar{\alpha}_{w}-t\|\tilde{s}_{w}+\tilde{x}_{w}-\bar{x}\|.\end{array}
\]
Equating the last two formulas gives 
\begin{eqnarray}
 & \begin{array}{c}
\frac{1}{2}t^{2}+t\|\tilde{s}_{w}+\tilde{x}_{w}-\bar{x}\|=\bar{\alpha}_{w}.\end{array}\label{eq:end-of-bar-alpha}
\end{eqnarray}
 Next, we have
\begin{eqnarray}
\begin{array}{c}
h_{w+1}(\tilde{x}_{w+1})+\frac{1}{2}\|\tilde{x}_{w+1}-\bar{x}\|^{2}\end{array} & \overset{\scriptsize{\mbox{Alg \ref{alg:basic-dual-ascent} line 9}}}{=} & \begin{array}{c}
\max\{h_{w},\tilde{h}_{w}\}(\tilde{x}_{w+1})+\frac{1}{2}\|\tilde{x}_{w+1}-\bar{x}\|^{2}\end{array}\nonumber \\
 & \overset{\eqref{eq:q-hat-lower-bdd},\eqref{eq:def-h-plus-w}}{\geq} & \begin{array}{c}
h_{w}^{+}(\tilde{x}_{w+1})\end{array}\label{eq:alpha-w-smaller}\\
 & \geq & \begin{array}{c}
h_{w}^{+}(\tilde{x}_{w}-t\tilde{d}_{w})\end{array}\nonumber \\
 & \overset{\eqref{eq:def-h-plus-w}}{=} & \begin{array}{c}
h_{w}(\tilde{x}_{w}-t\tilde{d}_{w})+\frac{1}{2}\|\tilde{x}_{w}-t\tilde{d}_{w}-\bar{x}\|^{2}\end{array}\nonumber \\
 & \overset{\eqref{eq:t-square-term}}{=} & \begin{array}{c}
h_{w}(\tilde{x}_{w})+\frac{1}{2}\|\tilde{x}_{w}-\bar{x}\|^{2}+\frac{1}{2}t^{2}.\end{array}\nonumber 
\end{eqnarray}
 The formulas \eqref{eq:end-of-bar-alpha} and \eqref{eq:alpha-w-smaller}
imply the first part of our lemma. Next, let $t_{2}$ be the positive
root of 
\begin{equation}
\begin{array}{c}
\frac{1}{2}t_{2}^{2}+t_{2}\|\tilde{s}_{w}+\tilde{x}_{w}-\bar{x}\|=\alpha_{w}.\end{array}\label{eq:quad-form}
\end{equation}
Since $\alpha_{w}\leq\bar{\alpha}_{w}$, we have $t_{2}\leq t$. Recalling
the definition of $\alpha_{w}$ in \eqref{eq:alpha-w}, we have 
\[
\begin{array}{c}
\alpha_{w+1}\overset{\eqref{eq:alpha-w},\eqref{eq:alpha-w-smaller}}{\leq}\alpha_{w}-\frac{1}{2}t^{2}\leq\alpha_{w}-\frac{1}{2}t_{2}^{2}\overset{\eqref{eq:quad-form}}{=}t_{2}\|\tilde{s}_{w}+\tilde{x}_{w}-\bar{x}\|,\end{array}
\]
or $\frac{\alpha_{w+1}}{\|\tilde{s}_{w}+\tilde{x}_{w}-\bar{x}\|}\leq t_{2}$.
Substituting this into \eqref{eq:quad-form} gives \eqref{eq:target-quad}
as needed.
\end{proof}
\begin{rem}
It is clear that $h_{0}(\cdot)$ in Algorithm \ref{alg:basic-dual-ascent}
can be defined as an affine function based on the evaluation of $f(x)$
and a subgradient in $\partial f(x)$ for some point $x$. In line
9, instead of $h_{w+1}(\cdot)$ defined there, one can use the maximum
of a number of affine functions like in the bundle method. We shall
only limit to the easy case of using one affine function to model
$h_{w+1}(\cdot)$ for pedagogical reasons.
\end{rem}

We need the following result proved in \cite{Beck_Tetruashvili_2013}
and \cite{Beck_alt_min_SIOPT_2015}.
\begin{lem}
\label{lem:seq-conv-rate}(Sequence convergence rate) Let $\alpha>0$.
Suppose the sequence of nonnegative numbers $\{a_{k}\}_{k=0}^{\infty}$
is such that 
\[
a_{k}\geq a_{k+1}+\alpha a_{k+1}^{2}\mbox{ for all }k\in\{1,2,\dots\}.
\]
\end{lem}

\begin{enumerate}
\item \cite[Lemma 6.2]{Beck_Tetruashvili_2013} If furthermore, $\begin{array}{c}
a_{1}\leq\frac{1.5}{\alpha}\mbox{ and }a_{2}\leq\frac{1.5}{2\alpha}\end{array}$, then 
\[
\begin{array}{c}
a_{k}\leq\frac{1.5}{\alpha k}\mbox{ for all }k\in\{1,2,\dots\}.\end{array}
\]
\item \cite[Lemma 3.8]{Beck_alt_min_SIOPT_2015} For any $k\geq2$, 
\[
\begin{array}{c}
a_{k}\leq\max\left\{ \left(\frac{1}{2}\right)^{(k-1)/2}a_{0},\frac{4}{\alpha(k-1)}\right\} .\end{array}
\]
In addition, for any $\epsilon>0$, if 
\[
\begin{array}{c}
\begin{array}{c}
k\geq\max\left\{ \frac{2}{\ln(2)}[\ln(a_{0})+\ln(1/\epsilon)],\frac{4}{\alpha\epsilon}\right\} +1,\end{array}\end{array}
\]
then $a_{k}\leq\epsilon$. 
\end{enumerate}
Theorem \ref{thm:basic-conv-rate} shows that Algorithm \ref{alg:basic-dual-ascent}
has convergence rates consistent with standard first order methods.
\begin{thm}
\label{thm:basic-conv-rate}(Convergence rate) Suppose Algorithm \ref{alg:basic-dual-ascent}
is used to solve \eqref{eq:small-pblm}, and $\dom(f)=\mathbb{R}^{m}$.

\begin{enumerate}
\item There is a $O(1/w)$ convergence rate.
\item If in addition, $\nabla f(\cdot)$ is Lipschitz with constant $L_{1}$,
then there is a linear rate of convergence. 
\end{enumerate}
\end{thm}

\begin{proof}
Recall that $h_{w}(\cdot)\leq f(\cdot)$ from Lemma \ref{lem:h-w-leq-f},
so 
\begin{eqnarray*}
\begin{array}{c}
f(x^{*})+\frac{1}{2}\|x^{*}-\bar{x}\|^{2}\end{array} & \geq & \begin{array}{c}
h_{w}(x^{*})+\frac{1}{2}\|x^{*}-\bar{x}\|^{2}\end{array}\\
 & \overset{\scriptsize{\mbox{Alg \ref{alg:basic-dual-ascent} line 2}}}{\geq} & \begin{array}{c}
h_{w}(\tilde{x}_{w})+\frac{1}{2}\|\tilde{x}_{w}-\bar{x}\|^{2}+\frac{1}{2}\|x^{*}-\tilde{x}_{w}\|^{2}.\end{array}\\
\begin{array}{c}
\Rightarrow\alpha_{w}\end{array} & \overset{\eqref{eq:alpha-w}}{\geq} & \begin{array}{c}
\frac{1}{2}\|x^{*}-\tilde{x}_{w}\|^{2}.\end{array}
\end{eqnarray*}
Therefore, $\|\tilde{x}_{w}-x^{*}\|\leq\sqrt{2\alpha_{w}}$. We have
$0\in\partial[f(\cdot)+\frac{1}{2}\|\cdot-\bar{x}\|^{2}](x^{*})$
and $\tilde{s}_{w}+\tilde{x}_{w}-\bar{x}\in\partial[f(\cdot)+\frac{1}{2}\|\cdot-\bar{x}\|^{2}](\tilde{x}_{w})$. 

\textbf{\uline{Case 1: General case}} 

Since $\dom(f)=\mathbb{R}^{m}$ and $\tilde{x}_{w}$ lies in a compact
set, there is some constant $L>0$ such that $\|\tilde{s}_{w}+\tilde{x}_{w}-\bar{x}\|\leq L$.
Hence 
\[
\begin{array}{c}
\frac{1}{2L^{2}}\alpha_{w+1}^{2}+\alpha_{w+1}\leq\frac{1}{2\|\tilde{s}_{w}+\tilde{x}_{w}-\bar{x}\|^{2}}\alpha_{w+1}^{2}+\alpha_{w+1}\overset{\eqref{eq:target-quad}}{\leq}\alpha_{w}.\end{array}
\]

This recurrence together with Lemma \ref{lem:seq-conv-rate} gives
us the $O(1/w)$ convergence rate we need. 

\textbf{\uline{Case 2: Smooth case.}}\textbf{ } Let $L_{1}$ be
the Lipschitz constant on the gradient of $f(\cdot)+\frac{1}{2}\|\cdot-\bar{x}\|^{2}$.
We then have $\|\tilde{s}_{w}+\tilde{x}_{w}-\bar{x}\|\leq L_{1}\|\tilde{x}_{w}-x^{*}\|\leq L_{1}\sqrt{2\alpha_{w}}$.
 Using the formula \eqref{eq:target-quad} gives us $\frac{1}{4L_{1}\alpha_{w}}\alpha_{w+1}^{2}+\alpha_{w+1}\leq\alpha_{w}$,
or 
\[
\begin{array}{c}
\frac{1}{4L_{1}}\left(\frac{\alpha_{w+1}}{\alpha_{w}}\right)^{2}+\frac{\alpha_{w+1}}{\alpha_{w}}\leq1.\end{array}
\]
This gives us the linear convergence as needed.
\end{proof}
\begin{rem}
\label{rem:min-max-of-2-quads}(Minimizing \eqref{eq:max-quad}) The
quadratic program can be solved easily by noting that the minimizer
must be a minimizer of one of the problems 
\begin{eqnarray*}
 &  & \begin{array}{c}
\underset{x}{\min}\,\tilde{h}_{w}(x)+\frac{1}{2}\|x-\bar{x}\|^{2},\text{ }\underset{x}{\min}\,h_{w}(x)+\frac{1}{2}\|x-\bar{x}\|^{2},\end{array}\\
 & \text{ or} & \begin{array}{c}
\underset{x}{\min}\{h_{w}(x)+\frac{1}{2}\|x-\bar{x}\|^{2}:\tilde{h}_{w}(x)=h_{w}(x)\},\end{array}
\end{eqnarray*}
all of which are rather easy to solve. 
\end{rem}

We now state a proposition that will be useful for the proof of convergence.
\begin{prop}
\label{prop:P-D-of-one-block}Consider the problem 
\[
\begin{array}{c}
\underset{x\in\mathbb{R}^{m}}{\min}\frac{1}{2}\|\bar{x}-x\|^{2}+f(x),\end{array}
\]
and its corresponding (Fenchel) dual 
\[
\begin{array}{c}
\underset{y\in\mathbb{R}^{m}}{\max}-\frac{1}{2}\|\bar{x}-y\|^{2}+\frac{1}{2}\|\bar{x}\|^{2}-f^{*}(y).\end{array}
\]
The optimal solutions $x^{*}$ and $y^{*}$ are related by $x^{*}+y^{*}=\bar{x}$,
and strong duality holds.
\end{prop}

\begin{proof}
This result can be seen to be Moreau's decomposition theorem.
\end{proof}
In view of Proposition \ref{prop:P-D-of-one-block}, we now explain
that Algorithm \ref{alg:basic-dual-ascent} can be interpreted as
a dual ascent algorithm. We can see that Algorithm \ref{alg:basic-dual-ascent}
finds $h_{w}(\cdot)$ for $w=0,1,\dots$ such that $h_{w}^{*}(\cdot)\geq f^{*}(\cdot)$
and dual iterates $\{\tilde{y}_{w}\}_{w=0}^{\infty}$ so that $\{-h_{w}^{*}(\tilde{y}_{w})+\frac{1}{2}\|\bar{x}\|^{2}-\frac{1}{2}\|\tilde{y}_{w}-\bar{x}\|^{2}\}_{w=0}^{\infty}$
is a monotonically nondecreasing sequence that converges to $-f^{*}(y^{*})+\frac{1}{2}\|\bar{x}\|^{2}-\frac{1}{2}\|y^{*}-\bar{x}\|^{2}$,
where $y^{*}$ is the optimal dual variable for \eqref{eq:small-pblm}.
This interpretation shall be exploited in our subdifferentiable distributed
Dykstra's algorithm. 

\section{\label{sec:main-alg}Deterministic Distributed Asynchronous Dykstra
Algorithm}

We now proceed to integrate the dual ascent algorithm in Section \ref{sec:First-alg}
into the distributed Dykstra's algorithm for the problem \eqref{eq:distrib-dyk-primal}.
We partition the vertex set $\mathcal{V}$ as the disjoint union \emph{$\mathcal{V}=\mathcal{V}_{1}\cup\mathcal{V}_{2}\cup\mathcal{V}_{3}\cup\mathcal{V}_{4}$
}so that 
\begin{itemize}
\item $f_{i}(\cdot)$ are proximable functions for all $i\in\mathcal{V}_{1}$.
\item $f_{i}(\cdot)$ are indicator functions of closed convex sets for
all $i\in\mathcal{V}_{2}$.
\item $f_{i}(\cdot)$ are proximable functions such that $\dom(f_{i})=\mathbb{R}^{m}$
for all $i\in\mathcal{V}_{3}$.
\item \emph{$f_{i}(\cdot)$ }are subdifferentiable functions (i.e., a subgradient
is easy to obtain) such that $\dom(f_{i})=\mathbb{R}^{m}$ for all
$i\in\mathcal{V}_{4}$. 
\end{itemize}
We had the 3 sets $\mathcal{V}_{1}$, $\mathcal{V}_{2}$ and $\mathcal{V}_{3}$
in \cite{Pang_Dyk_spl}. In principle, the vertices in $\mathcal{V}_{2}$
and $\mathcal{V}_{3}$ can be placed into $\mathcal{V}_{1}$. As explained
in \cite{Pang_Dyk_spl}, the advantage of separating $\mathcal{V}_{3}$
is that more than one function can be minimized at a time for vertices
in this set without affecting the proof of convergence (see Proposition
\ref{prop:control-growth}), and the advantage of separating $\mathcal{V}_{2}$
is that one can apply a greedy SHQP step in \cite{Pang_DBAP}. The
set $\mathcal{V}_{4}$ contains subdifferentiable functions, which
is the subject of this paper. 

To simplify calculations, we let $\mathbf{{v}}_{A}$, $\mathbf{{v}}_{H}$
and $\mathbf{{x}}$ be denoted by\begin{subequations}\label{eq_m:all_acronyms}
\begin{eqnarray}
\mathbf{{v}}_{H} & = & \sum_{((i,j),\bar{k})\in\bar{\mathcal{E}}}\mathbf{{z}}_{((i,j),\bar{k})}\label{eq:v-H-def}\\
\mathbf{{v}}_{A} & = & \mathbf{{v}}_{H}+\sum_{i\in\mathcal{V}}\mathbf{{z}}_{i}\label{eq:from-10}\\
\mathbf{{x}} & = & \bar{\mathbf{{x}}}-\mathbf{{v}}_{A}.\label{eq:x-from-v-A}
\end{eqnarray}
\end{subequations}Intuitively, $\mathbf{{v}}_{H}$ describes the
sum of the dual variables due to $H_{((i,j),\bar{k})}$ for all $((i,j),\bar{k})\in\bar{\mathcal{E}}$,
$\mathbf{{v}}_{A}$ is the sum of all dual variables, and $\mathbf{{x}}$
is the estimate of the primal variable. 

\subsection{\label{subsec:Partial-comm-prelim}Partial communication of data}

One insight that we point out in this paper is that Algorithm \ref{alg:Ext-Dyk}
supports the partial communication of data. We lay down the foundations
of the parts of Algorithm \ref{alg:Ext-Dyk} relevant for this insight. 

Let $D\subset[\mathbb{R}^{m}]^{|\mathcal{V}|}$ be the diagonal set
defined by 
\begin{equation}
D:=\{\mathbf{{x}}\in[\mathbb{R}^{m}]^{|\mathcal{V}|}:\mathbf{{x}}_{1}=\mathbf{{x}}_{2}=\cdots=\mathbf{{x}}_{|\mathcal{V}|}\}.\label{eq:diagonal-set}
\end{equation}
With the definition of the hyperplanes $H_{((i,j),\bar{k})}$ in \eqref{eq:def-halfspaces}
and $\mathcal{G}=(\mathcal{V},\mathcal{E})$ being a connected graph,
we have 
\begin{equation}
\bigcap_{((i,j),\bar{k})\in\bar{\mathcal{E}}}H_{((i,j),\bar{k})}=D\mbox{ and }\sum_{((i,j),\bar{k})\in\bar{\mathcal{E}}}H_{((i,j),\bar{k})}^{\perp}=D^{\perp}=\left\{ \mathbf{{z}}\in[\mathbb{R}^{m}]^{|\mathcal{V}|}:\sum_{i\in\mathcal{V}}\mathbf{{z}}_{i}=0\right\} .\label{eq:D-and-D-perp}
\end{equation}

\begin{prop}
\label{prop:E-connects-V}Suppose $\mathcal{G}=(\mathcal{V},\mathcal{E})$
is a connected graph. Let $H_{((i,j),\bar{k})}$ be the hyperplane
\eqref{eq:def-halfspaces}. Let $\bar{\mathcal{E}}'$ be a subset
of $\bar{\mathcal{E}}$. The following conditions are equivalent:

\begin{enumerate}
\item $\cap_{((i,j),\bar{k})\in\bar{\mathcal{E}}'}H_{((i,j),\bar{k})}=D$
\item $\sum_{((i,j),\bar{k})\in\mathcal{\bar{E}}'}H_{((i,j),\bar{k})}^{\perp}=D^{\perp}.$
\item For each $\bar{k}\in\{1,\dots,m\}$, the graph $\mathcal{G}'=(\mathcal{V},\mathcal{E}_{\bar{k}}')$
is connected, where $\mathcal{E}_{\bar{k}}':=\{(i,j)\in\mathcal{E}:((i,j),\bar{k})\in\bar{\mathcal{E}}'\}$. 
\end{enumerate}
\end{prop}

\begin{proof}
The equivalence between (1) and (3) is easy, and the equivalence between
(1) and (2) is simple linear algebra. 
\end{proof}
\begin{defn}
\label{def:E-connects-V}We say that\emph{ $\bar{\mathcal{E}}'\subset\bar{\mathcal{E}}$
connects $\mathcal{V}$ }if any of the equivalent properties in Proposition
\ref{prop:E-connects-V} is satisfied.
\end{defn}

\begin{rem}
\label{rem:partial-comms-change}(Change from \cite{Pang_Dist_Dyk})
The change in this paper from \cite{Pang_Dist_Dyk} is that each hyperplane
$H_{((i,j),\bar{k})}$ is now of codimension 1. In \cite{Pang_Dist_Dyk},
we defined the hyperplanes $H_{(i,j)}:=\{\mathbf{{x}}\in[\mathbb{R}^{m}]^{|\mathcal{V}|}:\mathbf{{x}}_{i}=\mathbf{{x}}_{j}\}$
of codimension $m$ which are indexed by $(i,j)\in\mathcal{E}$ instead.
The advantage of introducing the additional variables is that we can
have a partial transfer of the data between two vertices rather than
a full transfer. This will be elaborated in Example \ref{exa:partial-comms}.
\end{rem}

\begin{lem}
\label{lem:express-v-as-sum}(Expressing $v$ as a sum) Recall the
definitions of $D$ and $H_{((i,j),\bar{k})}$ in \eqref{eq:diagonal-set}
and \eqref{eq:def-halfspaces}. There is a $C_{1}>0$ such that for
all $\mathbf{{v}}\in D^{\perp}$ and $\bar{\mathcal{E}}'\subset\bar{\mathcal{E}}$
such that $\bar{\mathcal{E}}'$ connects $\mathcal{V}$, we can find
$\mathbf{{z}}_{((i,j),\bar{k})}\in H_{((i,j),\bar{k})}^{\perp}$ for
all $((i,j),\bar{k})\in\bar{\mathcal{E}}'$ such that $\sum_{((i,j),\bar{k})\in\bar{\mathcal{E}}'}\mathbf{{z}}_{((i,j),\bar{k})}=\mathbf{{v}}$
and $\|\mathbf{{z}}_{((i,j),\bar{k})}\|\leq C_{1}\|\mathbf{{v}}\|$
for all $((i,j),\bar{k})\in\bar{\mathcal{E}}'$.
\end{lem}

\begin{proof}
This is elementary linear algebra. We refer to \cite{Pang_Dist_Dyk}
for a proof of a similar result. 
\end{proof}

\subsection{Algorithm description and preliminaries}

In this subsection, we present Algorithm \ref{alg:Ext-Dyk} below
and recall some of the results that were presented in \cite{Pang_Dist_Dyk}
that are necessary for further discussions.

Recall that in the one node case in Section \ref{sec:First-alg},
the subdifferentiable function $f_{i}(\cdot)$ is handled using lower
approximates. In addition to \eqref{eq:Dykstra-dual-defn}, we need
to consider the function
\begin{equation}
\begin{array}{c}
F_{n,w}(\{\mathbf{{z}}_{\alpha}\}_{\alpha\in\bar{\mathcal{E}}\cup\mathcal{V}}):=-\frac{1}{2}\bigg\|\bar{\mathbf{{x}}}-\underset{\alpha\in\bar{\mathcal{E}}\cup\mathcal{V}}{\sum}\mathbf{{z}}_{\alpha}\bigg\|^{2}+\frac{1}{2}\|\bar{\mathbf{{x}}}\|^{2}-\underset{\alpha\in\bar{\mathcal{E}}\cup\mathcal{V}}{\sum}\mathbf{f}_{\alpha,n,w}^{*}(\mathbf{{z}}_{\alpha}),\end{array}\label{eq:Dykstra-dual-defn-1}
\end{equation}
where for all $n\geq1$ and $w\in\{1,\dots,\bar{w}\}$, $\mathbf{f}_{\alpha,n,w}:[\mathbb{R}^{m}]^{|\mathcal{V}|}\to\mathbb{R}$
satisfies \begin{subequations}\label{eq_m:h-a-n-w} 
\begin{eqnarray}
\mathbf{f}_{\alpha,n,w}(\cdot) & = & \mathbf{f}_{\alpha}(\cdot)\mbox{ for all }\alpha\in[\bar{\mathcal{E}}\cup\mathcal{V}]\backslash\mathcal{V}_{4}\label{eq:h-a-n-w-eq-h-a}\\
\mbox{ and }\mathbf{f}_{\alpha,n,w}(\cdot) & \leq & \mathbf{f}_{\alpha}(\cdot)\mbox{ for all }\alpha\in\mathcal{V}_{4}.\label{eq:h-a-n-w-lesser}
\end{eqnarray}
\end{subequations} So $F_{n,w}(\cdot)\leq F(\cdot)$. We present
Algorithm \vref{alg:Ext-Dyk}. 

\begin{algorithm}[!h]
\begin{lyxalgorithm}
\label{alg:Ext-Dyk}(Distributed Dykstra's algorithm) Consider the
problem \eqref{eq:Dyk-primal} along with the associated dual problem
\eqref{eq:dual-fn}.

Let $\bar{w}$ be a positive integer. Let $C_{1}>0$ satisfy Lemma
\ref{lem:express-v-as-sum}. For each $\alpha\in[\bar{\mathcal{E}}\cup\mathcal{V}]\backslash\mathcal{V}_{4}$,
$n\geq1$ and $w\in\{1,\dots,\bar{w}\}$, let $\mathbf{f}_{\alpha,n,w}:[\mathbb{R}^{m}]^{|\mathcal{V}|}\to\mathbb{R}$
be as defined in \eqref{eq_m:h-a-n-w}. Our distributed Dykstra's
algorithm is as follows:

01$\quad$Let 

\begin{itemize}
\item $\mathbf{{z}}_{i}^{1,0}\in[\mathbb{R}^{m}]^{|\mathcal{V}|}$ be a
starting dual vector for $\mathbf{f}_{i}(\cdot)$ for each $i\in\mathcal{V}$
so that $[\mathbf{{z}}_{i}^{1,0}]_{j}=0\in\mathbb{R}^{m}$ for all
$j\in\mathcal{V}\backslash\{i\}$. 
\item $\mathbf{{v}}_{H}^{1,0}\in D^{\perp}$ be a starting dual vector for
\eqref{eq:dual-fn}.

\begin{itemize}
\item Note: $\{\mathbf{{z}}_{((i,j),\bar{k})}^{n,0}\}_{((i,j),\bar{k})\in\bar{\mathcal{E}}}$
is defined through $\mathbf{{v}}_{H}^{n,0}$ in \eqref{eq_m:resetted-z-i-j}.
\end{itemize}
\item Let $\mathbf{{x}}^{1,0}$ be $\mathbf{{x}}^{1,0}=\bar{\mathbf{{x}}}-\mathbf{{v}}_{H}^{1,0}-\sum_{i\in\mathcal{V}}\mathbf{{z}}_{i}^{1,0}$.
\end{itemize}
02$\quad$For each $i\in\mathcal{V}_{4}$, let $\mathbf{f}_{i,1,0}:[\mathbb{R}^{m}]^{|\mathcal{V}|}\to\mathbb{R}$
be a function such that $\mathbf{f}_{i,1,0}(\cdot)\leq\mathbf{f}_{i}(\cdot)$

03$\quad$For $n=1,2,\dots$

04$\quad$$\quad$\textup{Let $\bar{\mathcal{E}}_{n}\subset\bar{\mathcal{E}}$
be such that $\bar{\mathcal{E}}_{n}$ connects $\mathcal{V}$ in the
sense of Definition \ref{def:E-connects-V}.}

05$\quad$$\quad$\textup{Define $\{\mathbf{{z}}_{((i,j),\bar{k})}^{n,0}\}_{((i,j),\bar{k})\in\bar{\mathcal{E}}}$
so that:\begin{subequations}\label{eq_m:resetted-z-i-j} 
\begin{eqnarray}
\mathbf{{z}}_{((i,j),\bar{k})}^{n,0} & = & 0\mbox{ for all }((i,j),\bar{k})\notin\mathcal{\bar{E}}_{n}\label{eq:reset-z-i-j-1}\\
\mathbf{{z}}_{((i,j),\bar{k})}^{n,0} & \in & H_{((i,j),\bar{k})}^{\perp}\mbox{ for all }((i,j),\bar{k})\in\bar{\mathcal{E}}\label{eq:reset-z-i-j-2}\\
\|\mathbf{{z}}_{((i,j),\bar{k})}^{n,0}\| & \leq & C_{1}\|\mathbf{{v}}_{H}^{n,0}\|\mbox{ for all }((i,j),\bar{k})\in\bar{\mathcal{E}}\label{eq:reset-z-i-j-3}\\
\mbox{ and }\sum_{((i,j),\bar{k})\in\bar{\mathcal{E}}}\mathbf{{z}}_{((i,j),\bar{k})}^{n,0} & = & \mathbf{{v}}_{H}^{n,0}.\label{eq:reset-z-i-j-4}
\end{eqnarray}
\end{subequations}}

$\quad$$\quad$\textup{(This is possible by Lemma \ref{lem:express-v-as-sum}.) }

06$\quad$$\quad$For $w=1,2,\dots,\bar{w}$

07$\quad$$\quad$$\quad$Choose a set $S_{n,w}\subset\bar{\mathcal{E}}_{n}\cup\mathcal{V}$
such that $S_{n,w}\neq\emptyset$. 

08$\quad$$\quad$$\quad$If $S_{n,w}\subset\mathcal{V}_{4}$, then 

09$\quad$$\quad$$\quad$$\quad$Apply Algorithm \ref{alg:subdiff-subalg}.

10$\quad$$\quad$$\quad$else

11$\quad$$\quad$$\quad$$\quad$Set $\mathbf{f}_{i,n,w}(\cdot):=\mathbf{f}_{i,n,w-1}(\cdot)$
for all $i\in\mathcal{V}_{4}$.

12$\quad$$\quad$$\quad$$\quad$Define $\{\mathbf{{z}}_{\alpha}^{n,w}\}_{\alpha\in S_{n,w}}$
by 
\begin{equation}
\{\mathbf{{z}}_{\alpha}^{n,w}\}_{\alpha\in S_{n,w}}=\underset{z_{\alpha},\alpha\in S_{n,w}}{\arg\min}\frac{1}{2}\left\Vert \bar{\mathbf{{x}}}-\sum_{\alpha\notin S_{n,w}}\mathbf{{z}}_{\alpha}^{n,w-1}-\sum_{\alpha\in S_{n,w}}\mathbf{{z}}_{\alpha}\right\Vert ^{2}+\sum_{\alpha\in S_{n,w}}\mathbf{f}_{\alpha,n,w}^{*}(\mathbf{{z}}_{\alpha}).\label{eq:Dykstra-min-subpblm}
\end{equation}

13$\quad$$\quad$$\quad$end if 

14$\quad$$\quad$$\quad$Set $\mathbf{{z}}_{\alpha}^{n,w}:=\mathbf{{z}}_{\alpha}^{n,w-1}$
for all $\alpha\notin S_{n,w}$.

15$\quad$$\quad$End For 

16$\quad$$\quad$Let $\mathbf{{z}}_{i}^{n+1,0}=\mathbf{{z}}_{i}^{n,\bar{w}}$
for all $i\in\mathcal{V}$ and $\mathbf{{v}}_{H}^{n+1,0}=\mathbf{{v}}_{H}^{n,\bar{w}}=\sum_{((i,j),\bar{k})\in\bar{\mathcal{E}}}\mathbf{{z}}_{((i,j),\bar{k})}^{n,\bar{w}}$.

17$\quad$$\quad$Let $\mathbf{f}_{i,n+1,0}(\cdot)=\mathbf{f}_{i,n,\bar{w}}(\cdot)$
for all $i\in\mathcal{V}_{4}$.

18$\quad$End For 
\end{lyxalgorithm}

\end{algorithm}

Even though Algorithm \ref{alg:Ext-Dyk} is described so that each
node $i\in\mathcal{V}$ and $((i,j),\bar{k})\in\bar{\mathcal{E}}$
is associated with a dual variable $\mathbf{{z}}_{\alpha}\in[\mathbb{R}^{m}]^{|\mathcal{V}|}$,
we point out that the size of the dual variable $z_{\alpha}$ that
needs to be stored in each node and edge is small due to sparsity.
\begin{prop}
\label{prop:sparsity}(Sparsity of $\mathbf{{z}}_{\alpha}$) We have
$[\mathbf{{z}}_{i}^{n,w}]_{j}=0$ for all $j\in\mathcal{V}\backslash\{i\}$,
$n\geq1$ and $w\in\{0,1,\dots,\bar{w}\}$. Similarly, for all $n\geq1$,
$w\in\{0,1,\dots,\bar{w}\}$ and $(e,\bar{k})\in\bar{\mathcal{E}}$,
the vector $\mathbf{{z}}_{(e,\bar{k})}^{n,w}\in[\mathbb{R}^{m}]^{|\mathcal{V}|}$
satisfies $[[\mathbf{{z}}_{(e,\bar{k})}^{n,w}]_{i'}]_{k'}=0$ unless
$\bar{k}=k'$ and $i$ is an endpoint of $e$.
\end{prop}

\begin{proof}
The proof of this result is similar to the corresponding result in
\cite{Pang_Dist_Dyk}. The claim for $z_{i}^{n,w}$ relies on the
fact that $\mathbf{f}_{i,n,w}(\cdot)$ depends only on the $i$-th
component, and the claim for $\mathbf{{z}}_{(e,\bar{k})}^{n,w}$ relies
on the fact that $\mathbf{f}_{(e,\bar{k})}(\cdot)=\delta_{H_{(e,\bar{k})}^{\perp}}(\cdot)$,
with $H_{(e,\bar{k})}^{\perp}$ containing vectors that are zero in
all but 2 coordinates.
\end{proof}
Dykstra's algorithm is traditionally written in terms of solving for
the primal variable $x$. For completeness, we show the equivalence
between \eqref{eq:Dykstra-min-subpblm} and the primal minimization
problem. The proof is easily extended from \cite{Pang_Dyk_spl}. 
\begin{prop}
\label{prop:subproblems}(On solving \eqref{eq:Dykstra-min-subpblm})
If a minimizer $\{\mathbf{{z}}_{\alpha}^{n,w}\}_{\alpha\in S_{n,w}}$
for \eqref{eq:Dykstra-min-subpblm} exists, then the $x^{n,w}$ in
\eqref{eq:x-from-v-A} satisfies 
\begin{equation}
\begin{array}{c}
\mathbf{{x}}^{n,w}=\underset{\mathbf{{x}}\in[\mathbb{R}^{m}]^{|\mathcal{V}|}}{\arg\min}\underset{\alpha\in S_{n,w}}{\sum}\mathbf{f}_{\alpha,n,w}(\mathbf{{x}})+\frac{1}{2}\left\Vert \mathbf{{x}}-\left(\bar{\mathbf{{x}}}-\underset{\alpha\notin S_{n,w}}{\sum}\mathbf{{z}}_{\alpha}^{n,w}\right)\right\Vert ^{2}.\end{array}\label{eq:primal-subpblm}
\end{equation}
Conversely, if $\mathbf{{x}}^{n,w}$ solves \eqref{eq:primal-subpblm}
with the dual variables $\{\tilde{\mathbf{{z}}}_{\alpha}^{n,w}\}_{\alpha\in S_{n,w}}$
satisfying 
\begin{equation}
\begin{array}{c}
\tilde{\mathbf{{z}}}_{\alpha}^{n,w}\in\partial\mathbf{f}_{\alpha,n,w}(\mathbf{{x}}^{n,w})\mbox{ and }\mathbf{{x}}^{n,w}-\bar{\mathbf{{x}}}+\underset{\alpha\notin S_{n,w}}{\sum}\mathbf{{z}}_{\alpha}^{n,w}+\underset{\alpha\in S_{n,w}}{\sum}\tilde{\mathbf{{z}}}_{\alpha}^{n,w}=0,\end{array}\label{eq:primal-optim-cond}
\end{equation}
then $\{\tilde{\mathbf{{z}}}_{\alpha}^{n,w}\}_{\alpha\in S_{n,w}}$
solves \eqref{eq:Dykstra-min-subpblm}. 
\end{prop}

\begin{rem}
\label{rem:Irrelevance-of-z}(Irrelevance of $\mathbf{{z}}_{((i,j),\bar{k})}^{n,w}$)
In \cite{Pang_Dist_Dyk}, we explained that each node $i\in\mathcal{V}$
needs to keep track of just $[\bar{\mathbf{{x}}}-\mathbf{{v}}_{H}^{n,w}]_{i}\in\mathbb{R}^{m}$
and $[\mathbf{{z}}_{i}^{n,w}]_{i}\in\mathbb{R}^{m}$, and does not
have to keep track of any part of the vectors $\mathbf{{z}}_{((i,j),\bar{k})}^{n,w}\in\mathbb{R}^{m}$
for $((i,j),\bar{k})\in\bar{\mathcal{E}}$. The same is true for Algorithms
\ref{alg:Ext-Dyk} and \ref{alg:subdiff-subalg} here. The reason
for introducing $\mathbf{{z}}_{((i,j),\bar{k})}^{n,w}\in[\mathbb{R}^{m}]^{|\mathcal{V}|}$
is that the proof of the convergence result in Theorem \ref{thm:convergence}
needs \eqref{eq:biggest-formula}, which in turn needs the variables
$\mathbf{{z}}_{((i,j),\bar{k})}^{n,w}$. 
\end{rem}

\begin{example}
\label{exa:partial-comms}(Partial communication of data) Fix some
$(i,j)\in\mathcal{E}$ and some set $\bar{K}\subset\{1,\dots,m\}$.
Suppose the set $S_{n,w}$ is chosen to be $\{((i,j),\bar{k}):\bar{k}\in\bar{K}\}$.
Then $\mathbf{{x}}^{n,w}$ is obtained from \eqref{eq:primal-subpblm},
which tells that $\mathbf{{x}}^{n,w}$ is the projection of $[\bar{\mathbf{{x}}}-\sum_{\alpha\notin S_{n,w}}\mathbf{{z}}_{\alpha}^{n,w-1}]$
onto $\cap_{((i,j),\bar{k}):\bar{k}\in\bar{K}}H_{((i,j),\bar{k})}$.
Since $H_{((i,j),\bar{k})}$ are all affine spaces with normals $\mathbf{{z}}_{((i,j),\bar{k})}^{n,w-1}$,
$\mathbf{{x}}^{n,w}$ is also the projection of 
\[
\begin{array}{c}
\bar{\mathbf{{x}}}-\underset{\alpha\notin S_{n,w}}{\sum}\mathbf{{z}}_{\alpha}^{n,w-1}-\underset{\alpha\in S_{n,w}}{\sum}\mathbf{{z}}_{\alpha}^{n,w-1},\end{array}
\]
or $\mathbf{{x}}^{n,w-1}$, onto $\cap_{((i,j),\bar{k}):\bar{k}\in\bar{K}}H_{((i,j),\bar{k})}$.
This gives 
\[
[[\mathbf{{x}}^{n,w}]_{i'}]_{k'}=\begin{cases}
\frac{1}{2}([[\mathbf{{x}}^{n,w-1}]_{i}]_{\bar{k}}+[[\mathbf{{x}}^{n,w-1}]_{j}]_{\bar{k}}) & \mbox{ if }i'\in\{i,j\}\mbox{ and }\bar{k}\in\bar{K}\\{}
[[\mathbf{{x}}^{n,w-1}]_{i'}]_{k'} & \mbox{ otherwise.}
\end{cases}
\]
As mentioned in Remark \ref{rem:Irrelevance-of-z}, there is no need
to keep track of the dual variables $\mathbf{{z}}_{((i,j),\bar{k})}^{n,w}$
to run Algorithm \ref{alg:Ext-Dyk}. So the larger $\bar{K}$ is,
the more variables are updated. Thus in Algorithm \ref{alg:Ext-Dyk},
computations can be performed continuously even when not all the data
is communicated. In other words, communications will not be a bottleneck
for Algorithm \ref{alg:Ext-Dyk}.
\end{example}

\subsection{Subroutine for subdifferentiable functions}

If $\mathcal{V}_{4}=\emptyset$, then Algorithm \ref{alg:Ext-Dyk}
corresponds mostly to the algorithm in \cite{Pang_Dist_Dyk} because
there are no subdifferentiable functions. In this subsection, we present
and derive Algorithm \ref{alg:subdiff-subalg}, which is a subroutine
within Algorithm \ref{alg:Ext-Dyk} to handle subdifferentiable functions. 

We state some notation necessary for further discussions. For any
$\alpha\in\bar{\mathcal{E}}\cup\mathcal{V}$ and $n\in\{1,2,\dots\}$,
let $p(n,\alpha)$ be 
\[
p(n,\alpha):=\max\{w':w'\leq\bar{w},\alpha\in S_{n,w'}\}.
\]
In other words, $p(n,\alpha)$ is the index $w'$ such that $\alpha\in S_{n,w'}$
but $\alpha\notin S_{n,k}$ for all $k\in\{w'+1,\dots,\bar{w}\}$.
It follows from line 14 in Algorithm \ref{alg:Ext-Dyk} that 
\begin{equation}
\mathbf{{z}}_{\alpha}^{n,p(n,\alpha)}=\mathbf{{z}}_{\alpha}^{n,p(n,\alpha)+1}=\cdots=\mathbf{{z}}_{\alpha}^{n,\bar{w}}\mbox{ for all }\alpha\in\bar{\mathcal{E}}\cup\mathcal{V}.\label{eq:stagnant-indices}
\end{equation}
Moreover, $((i,j),\bar{k})\notin\bar{\mathcal{E}}_{n}$ implies $((i,j),\bar{k})\notin S_{n,w}$
for all $w\in\{1,\dots,\bar{w}\}$, so 
\begin{equation}
0\overset{\scriptsize\eqref{eq:reset-z-i-j-1}}{=}\mathbf{{z}}_{((i,j),\bar{k})}^{n,0}=\mathbf{{z}}_{((i,j),\bar{k})}^{n,1}=\cdots=\mathbf{{z}}_{((i,j),\bar{k})}^{n,\bar{w}}\mbox{ for all }((i,j),\bar{k})\in\bar{\mathcal{E}}\backslash\bar{\mathcal{E}}_{n}.\label{eq:zero-indices}
\end{equation}

We present Algorithm \vref{alg:subdiff-subalg}. 

\begin{algorithm}[!h]
\begin{lyxalgorithm}
\label{alg:subdiff-subalg}(Subalgorithm for subdifferentiable functions)
This algorithm is run when line 9 of Algorithm \ref{alg:Ext-Dyk}
is reached. Suppose $S_{n,w}\subset\mathcal{V}_{4}$ and Assumption
\ref{assu:to-start-subalg} holds.

01 For each $i\in S_{n,w}$ 

02 $\quad$For $\tilde{f}_{i,n,w-1}(\cdot)$ defined in \eqref{eq:linearize-f-i-n-w},
consider 
\begin{equation}
\begin{array}{c}
\underset{x\in\mathbb{R}^{m}}{\min}\left[\frac{1}{2}\|[\bar{\mathbf{{x}}}-\mathbf{{v}}_{H}^{n,w-1}]_{i}-x\|^{2}+\max\{f_{i,n,w-1},\tilde{f}_{i,n,w-1}\}(x)\right],\end{array}\label{eq:alg-primal-subpblm}
\end{equation}

03 $\quad$Let the primal and dual solutions of \eqref{eq:alg-primal-subpblm}
be $x_{i}^{+}$ and $z_{i}^{+}$ 

04 $\quad$Define $f_{i,n,w}:\mathbb{R}^{m}\to\mathbb{R}$ to be the
affine function 
\begin{equation}
f_{i,n,w}(x):=f_{i,n,w-1}(x_{i}^{+})+\langle x-x_{i}^{+},[\bar{\mathbf{{x}}}-\mathbf{{v}}_{H}^{n,w-1}]_{i}-x_{i}^{+}\rangle.\label{eq:def-f-i-n-w}
\end{equation}

05 $\quad$In other words, $f_{i,n,w}(\cdot)$ is chosen such that
the 

$\qquad\qquad$primal and dual optimizers to \eqref{eq:alg-primal-subpblm}
coincide with that of 
\begin{equation}
\begin{array}{c}
\underset{x\in\mathbb{R}^{m}}{\min}\left[\frac{1}{2}\|[\bar{\mathbf{{x}}}-\mathbf{{v}}_{H}^{n,w-1}]_{i}-x\|^{2}+f_{i,n,w}(x)\right].\end{array}\label{eq:finw-design}
\end{equation}

06 $\quad$Define the function $\mathbf{f}_{i,n,w}:[\mathbb{R}^{m}]^{|\mathcal{V}|}\to\mathbb{R}$
and 

$\qquad\qquad$the dual vector $z_{i}^{n,w}\in[\mathbb{R}^{m}]^{|\mathcal{V}|}$
to be 
\begin{equation}
\mathbf{f}_{i,n,w}(\mathbf{{x}}):=f_{i,n,w}(\mathbf{{x}}_{i})\mbox{ and }[\mathbf{{z}}_{i}^{n,w}]_{j}:=\begin{cases}
\mathbf{{z}}_{i}^{+} & \mbox{ if }j=i\\
0 & \mbox{ if }j\neq i.
\end{cases}\label{eq:def-z-i}
\end{equation}

07 End for 

08 For all $i\in\mathcal{V}_{4}\backslash S_{n,w}$, $\mathbf{f}_{i,n,w}(\cdot)=\mathbf{f}_{i,n,w-1}(\cdot)$.
\end{lyxalgorithm}

\end{algorithm}
We make three assumptions that will be needed for the proof of convergence
of Theorem \ref{thm:convergence}. 
\begin{assumption}
\label{assu:to-start-subalg}(Start of Algorithm \ref{alg:subdiff-subalg})
Recall that at the start of Algorithm \ref{alg:subdiff-subalg}, $S_{n,w}\subset\mathcal{V}_{4}$.
We make three assumptions.

\begin{enumerate}
\item Suppose $(n,w)$ is such that $w>1$ and $S_{n,w}\subset\mathcal{V}_{4}$
so that Algorithm \ref{alg:subdiff-subalg} is invoked. Then for all
$i\in S_{n,w}$, $[\mathbf{{z}}_{i}^{n,w-1}]_{i}\in\mathbb{R}^{m}$
is the optimizer to the problem 
\begin{equation}
\begin{array}{c}
\underset{z\in\mathbb{R}^{m}}{\min}\frac{1}{2}\|[\bar{\mathbf{{x}}}-\mathbf{{v}}_{H}^{n,w-1}]_{i}-z\|^{2}+f_{i,n,w-1}^{*}(z).\end{array}\label{eq:multi-node-start}
\end{equation}
In other words, suppose $w_{i}\geq1$ is the largest $w'$ such that
$i\in S_{n,w'}$ and $i\notin S_{n,\tilde{w}}$ for all $\tilde{w}\in\{w'+1,w'+2,\dots,w-1\}$.
Then for all $\tilde{w}\in\{w_{i}+1,\dots,w-1\}$, $(e,\bar{k})\notin S_{n,\tilde{w}}$
if $i$ is an endpoint of $e$. 
\item Suppose that for all $i\in\mathcal{V}_{4}$ and $\tilde{w}\in\{p(n,i)+1,\dots,\bar{w}\}$,
$(e,\bar{k})\notin S_{n,\tilde{w}}$ if $i$ is an endpoint of $e$.
(This implies $\mathbf{x}_{i}^{n,p(n,i)}=\mathbf{x}_{i}^{n,\bar{w}}$.)
\item Suppose that $S_{n,1}=\mathcal{V}_{4}$ for all $n>1$.
\end{enumerate}
\end{assumption}

\begin{rem}
We need Assumption \ref{assu:to-start-subalg}(1) for Proposition
\ref{prop:quad-dec-case-2}, which is in turn needed for the proof
of Theorem \ref{thm:convergence}(i). We need Assumption \ref{assu:to-start-subalg}(2)
so that the analogue of Lemma \ref{lem:alpha-recurrs}(1) holds, which
in turn is used in the proof of Theorem \ref{thm:convergence}(iv).
Also, Assumption \ref{assu:to-start-subalg}(1) is seen to be satisfied
if $S_{n,w}\subset S_{n,w-1}$ if $S_{n,w}\subset\mathcal{V}_{4}$.

\begin{rem}
(On the problem \eqref{eq:alg-primal-subpblm}) Consider the case
where $S_{n,w}=\{i\}$ first, where $i\in\mathcal{V}_{4}$. If $i$
were in $\mathcal{V}\backslash\mathcal{V}_{4}$ instead, $\mathbf{{z}}_{i}^{n,w}$
is the minimizer of 
\begin{equation}
\begin{array}{c}
\underset{\mathbf{{z}}_{i}\in[\mathbb{R}^{m}]^{|\mathcal{V}|}}{\min}\frac{1}{2}\bigg\|\bar{\mathbf{{x}}}-\mathbf{{v}}_{H}^{n,w-1}-\underset{j\in\mathcal{V}\backslash\{i\}}{\sum}\mathbf{{z}}_{j}^{n,w-1}-\mathbf{{z}}_{i}\bigg\|^{2}+\mathbf{f}_{i}^{*}(\mathbf{{z}}_{i}).\end{array}\label{eq:block-dual}
\end{equation}
When $i\in\mathcal{V}_{4}$, we use $f_{i,n,w-1}(\cdot)$, where $f_{i,n,w-1}(\cdot)\leq f_{i}(\cdot)$,
instead of $f_{i}(\cdot)$. This gives $f_{i,n,w}^{*}(\cdot)\geq f_{i}^{*}(\cdot)$.
Instead of \eqref{eq:block-dual}, we now have 
\begin{equation}
\begin{array}{c}
\underset{\mathbf{{z}}_{i}\in[\mathbb{R}^{m}]^{|\mathcal{V}|}}{\min}\frac{1}{2}\bigg\|\bar{\mathbf{{x}}}-\mathbf{{v}}_{H}^{n,w-1}-\underset{j\in\mathcal{V}\backslash\{i\}}{\sum}\mathbf{{z}}_{j}^{n,w-1}-\mathbf{{z}}_{i}\bigg\|^{2}+\mathbf{f}_{i,n,w-1}^{*}(\mathbf{{z}}_{i}).\end{array}\label{eq:dual-of-approx}
\end{equation}
The dual of \eqref{eq:dual-of-approx} is (up to a constant independent
of $\mathbf{{x}}$) 
\begin{equation}
\begin{array}{c}
\underset{\mathbf{{x}}\in[\mathbb{R}^{m}]^{|\mathcal{V}|}}{\min}\frac{1}{2}\bigg\|\bar{\mathbf{{x}}}-\mathbf{{v}}_{H}^{n,w-1}-\underset{j\in\mathcal{V}\backslash\{i\}}{\sum}\mathbf{{z}}_{j}^{n,w-1}-\mathbf{{x}}\bigg\|^{2}+\mathbf{f}_{i,n,w-1}(\mathbf{{x}}).\end{array}\label{eq:dual-of-dual-subpblm}
\end{equation}
Since $\mathbf{{z}}_{i}^{n,w-1}\in[\mathbb{R}^{m}]^{|\mathcal{V}|}$
and $\mathbf{{z}}_{i}^{n,w}\in[\mathbb{R}^{m}]^{|\mathcal{V}|}$ are
such that the components in $\mathcal{V}\backslash\{i\}$ are all
zero by Proposition \ref{prop:sparsity}, the problem \eqref{eq:dual-of-dual-subpblm}
reduces to 
\begin{equation}
\begin{array}{c}
\underset{x\in\mathbb{R}^{m}}{\min}\frac{1}{2}\|[\bar{\mathbf{{x}}}-\mathbf{{v}}_{H}^{n,w-1}]_{i}-x\|^{2}+f_{i,n,w-1}(x).\end{array}\label{eq:unmod-alg-primal}
\end{equation}
Suppose that the minimizer of \eqref{eq:dual-of-approx} is $\mathbf{{z}}_{i}^{n,w-1}$,
which is the case when Assumption \ref{assu:to-start-subalg}(1) holds.
Then the minimizer of \eqref{eq:unmod-alg-primal} is $[\mathbf{{x}}^{n,w-1}]_{i}$,
which is also $[\bar{\mathbf{{x}}}-\mathbf{{v}}_{H}^{n,w-1}-\mathbf{{z}}_{i}^{n,w-1}]_{i}$
by \eqref{eq_m:all_acronyms}. Construct $\tilde{f}_{i,n,w-1}:\mathbb{R}^{m}\to\mathbb{R}$
by 
\begin{equation}
\tilde{f}_{i,n,w-1}(x):=f_{i}([\bar{\mathbf{{x}}}-\mathbf{{v}}_{H}^{n,w-1}-\mathbf{{z}}_{i}^{n,w-1}]_{i})+\langle s,x-[\bar{\mathbf{{x}}}-\mathbf{{v}}_{H}^{n,w-1}-\mathbf{{z}}_{i}^{n,w-1}]_{i}\rangle,\label{eq:linearize-f-i-n-w}
\end{equation}
where $s\in\partial f_{i}([\bar{\mathbf{{x}}}-\mathbf{{v}}_{H}^{n,w-1}-\mathbf{{z}}_{i}^{n,w-1}]_{i})$.
The primal problem that we now consider is \eqref{eq:alg-primal-subpblm}.
$\hfill\Delta$
\end{rem}

\end{rem}

\begin{rem}
(On the condition $S_{n,1}=V_{4}$) Throughout this paper, we assumed
$S_{n,1}=V_{4}$ in Assumption \ref{assu:to-start-subalg}. Algorithm
\ref{alg:Ext-Dyk} with this condition would not be truly asynchronous,
but it is relatively easy to enforce this condition. One way to enforce
this condition is to use a global clock. Another way to enforce this
condition is to use the sparsity of $\mathbf{z}_{\alpha}$ in Proposition
\ref{prop:sparsity}. Suppose that $\{S_{n,w}\}_{w=1}^{\bar{w}}$
is such that for all $i\in V_{4}$, $S_{n,w_{i}}=\{i\}$ for some
$w_{i}\in\{1,\dots,\bar{w}\}$. Suppose also that for all $i,j\in V_{4}$
such that $w_{i}<w_{j}$:

\begin{itemize}
\item [($\star$)]There are no $(e,k)\in\bar{E}$ such that $i$ and $j$
are the two endpoints of $e$ and $(e,k)\in S_{n,w'}$ for some $w'$
such that $w_{i}<w'<w_{j}$. 
\end{itemize}
If condition ($\star$) holds for some $i,j\in V_{4}$, then the sparsity
of $\mathbf{z}_{\alpha}^{n,w}$ implies that if we changed from $S_{n,w_{i}}=\{i\}$
and $S_{n,w_{j}}=\{j\}$ to $S_{n,w_{i}}=\{i,j\}$ and $S_{n,w_{j}}=\emptyset$,
then the iterates $\{\mathbf{x}^{n,w}\}_{w}$ obtained will remain
equivalent. It is possible to ensure ($\star$) for all $i,j\in V_{4}$
using a signal from a fixed node in $V$ propagated as computations
in the algorithm are carried out. 
\end{rem}

As mentioned in Remark \ref{rem:min-max-of-2-quads}, the problem
\eqref{eq:alg-primal-subpblm} is still easy to solve if $f_{i,n,w-1}(\cdot)$
and $\tilde{f}_{i,n,w-1}(\cdot)$ are affine functions with the known
parameters $[\bar{\mathbf{{x}}}-\mathbf{{v}}_{H}^{n,w-1}]_{i}$ and
$\mathbf{{z}}_{i}^{n,w-1}$. 

Next, for the primal optimizer $x_{i}^{+}$ defined in line 3 of Algorithm
\ref{alg:subdiff-subalg}, we can construct the affine function $f_{i,n,w}:\mathbb{R}^{m}\to\mathbb{R}$
to be such that 
\[
\begin{array}{c}
\underset{x\in\mathbb{R}^{m}}{\min}\frac{1}{2}\|[\bar{\mathbf{{x}}}-\mathbf{{v}}_{H}^{n,w-1}]_{i}-x\|^{2}+f_{i,n,w}(x)\end{array}
\]
has the same minimizer and objective value as \eqref{eq:alg-primal-subpblm}.
The function $f_{i,n,w}(\cdot)$ can be checked to be \eqref{eq:def-f-i-n-w}.
It is clear to see that $f_{i,n,w}(\cdot)\leq\max\{f_{i,n,w-1}(\cdot),\tilde{f}_{i,n,w-1}(\cdot)\}$.
Since both $f_{i,n,w-1}(\cdot)$ and $\tilde{f}_{i,n,w-1}(\cdot)$
are both by definition lower approximates of $f_{i}(\cdot)$, $f_{i,n,w}(\cdot)$
will also be a lower approximate of $f_{i}(\cdot)$. The function
$\mathbf{f}_{i,n,w}:[\mathbb{R}^{m}]^{|\mathcal{V}|}\to\mathbb{R}$
is constructed to be 
\[
\mathbf{f}_{i,n,w}(\mathbf{{x}})=f_{i,n,w}([\mathbf{{x}}]_{i}).
\]
The $\mathbf{{z}}_{i}^{n,w}\in[\mathbb{R}^{m}]^{|\mathcal{V}|}$ defined
by \eqref{eq:def-z-i} would be the optimal solution of the dual problem
\[
\begin{array}{c}
\underset{\mathbf{{z}}_{i}\in[\mathbb{R}^{m}]^{|\mathcal{V}|}}{\min}\frac{1}{2}\bigg\|\bar{\mathbf{{x}}}-\mathbf{{v}}_{H}^{n,w-1}-\underset{j\in\mathcal{V}\backslash\{i\}}{\sum}\mathbf{{z}}_{j}^{n,w-1}-\mathbf{{z}}_{i}\bigg\|^{2}+\mathbf{f}_{i,n,w}^{*}(\mathbf{{z}}_{i}).\end{array}
\]

\begin{rem}
(Similarities to the one node case) Note that the problem \eqref{eq:def-q}
corresponds to \eqref{eq:multi-node-start}, the function \eqref{eq:def-h-tilde-w}
to \eqref{eq:linearize-f-i-n-w}, the problem \eqref{eq:max-quad}
to \eqref{eq:alg-primal-subpblm}, and the function $h_{w+1}(\cdot)$
in line 9 of Algorithm \ref{alg:basic-dual-ascent} to \eqref{eq:def-f-i-n-w}. 
\end{rem}

One way to understand Proposition \ref{prop:P-D-of-one-block} is
to see that any change in the primal objective value gives the same
change in the dual objective value. We have the following result. 
\begin{prop}
\label{prop:quad-dec-case-2}Suppose $(n,w)$ is such that $w>1$
and $S_{n,w}\subset\mathcal{V}_{4}$ so that Algorithm \ref{alg:subdiff-subalg}
is run, and Assumption \ref{assu:to-start-subalg}(1) holds. Then
we have 
\begin{equation}
\begin{array}{c}
\frac{1}{2}\|\mathbf{{x}}^{n,w}-\mathbf{{x}}^{n,w-1}\|^{2}\leq F_{n,w}(\{\mathbf{{z}}_{\alpha}^{n,w}\}_{\alpha\in\bar{\mathcal{E}}\cup\mathcal{V}})-F_{n,w-1}(\{\mathbf{{z}}_{\alpha}^{n,w-1}\}_{\alpha\in\bar{\mathcal{E}}\cup\mathcal{V}}).\end{array}\label{eq:quad-dec-case-2}
\end{equation}
Similarly, if Assumption \ref{assu:to-start-subalg}(2) and (3) hold,
then 
\begin{equation}
\begin{array}{c}
\frac{1}{2}\|\mathbf{{x}}^{n,1}-\mathbf{{x}}^{n,0}\|^{2}\leq F_{n,1}(\{\mathbf{{z}}_{\alpha}^{n,1}\}_{\alpha\in\bar{\mathcal{E}}\cup\mathcal{V}})-F_{n,0}(\{\mathbf{{z}}_{\alpha}^{n,0}\}_{\alpha\in\bar{\mathcal{E}}\cup\mathcal{V}}).\end{array}\label{eq:quad-dec-case-3}
\end{equation}
\end{prop}

\begin{proof}
Recall Proposition \ref{prop:sparsity} on the sparsity of the $\mathbf{{z}}_{i}^{n,w}\in[\mathbb{R}^{m}]^{|\mathcal{V}|}$.
Recall that in line 3 of Algorithm \ref{alg:subdiff-subalg} the primal
and dual optimal solutions of \eqref{eq:alg-primal-subpblm} are $x_{i}^{+}$
and $z_{i}^{+}$. We can see that $x_{i}^{+}=[\mathbf{{x}}^{n,w}]_{i}$
and $z_{i}^{+}=[\mathbf{{z}}_{i}^{n,w}]_{i}$. Let the dual and primal
optimal solutions of \eqref{eq:unmod-alg-primal} be $z_{i}^{\circ}$
and $x_{i}^{\circ}$, which are $z_{i}^{\circ}=[\mathbf{{z}}_{i}^{n,w-1}]_{i}$
and $x_{i}^{\circ}=[\mathbf{{x}}^{n,w-1}]_{i}$ respectively. By Proposition
\ref{prop:P-D-of-one-block} and the forms of the problems \eqref{eq:alg-primal-subpblm}
and \eqref{eq:unmod-alg-primal}, we have $x_{i}^{+}+z_{i}^{+}=x_{i}^{\circ}+z_{i}^{\circ}$.
Thus $z_{i}^{+}-z_{i}^{\circ}=-(x_{i}^{+}-x_{i}^{\circ})$. In other
words, 
\begin{equation}
[\mathbf{{x}}^{n,w}-\mathbf{{x}}^{n,w-1}]_{i}=-[\mathbf{{z}}_{i}^{n,w}-\mathbf{{z}}_{i}^{n,w-1}]_{i}.\label{eq:for-beta-identity}
\end{equation}
Note that since $S_{n,w}\cap\bar{\mathcal{E}}=\emptyset$, $\mathbf{{v}}_{H}^{n,w}=\mathbf{{v}}_{H}^{n,w-1}$.
We have the following inequality chain, which we explain in a moment.
\begin{eqnarray}
 &  & \begin{array}{c}
f_{i,n,w}([\mathbf{{x}}^{n,w}]_{i})+\frac{1}{2}\|[\bar{\mathbf{{x}}}-\mathbf{{v}}_{H}^{n,w}]_{i}-[\mathbf{{x}}^{n,w}]_{i}\|^{2}\end{array}\label{eq:for-beta-chain}\\
 & = & \begin{array}{c}
f_{i,n,w-1}([\mathbf{{x}}^{n,w}]_{i})+\frac{1}{2}\|[\bar{\mathbf{{x}}}-\mathbf{{v}}_{H}^{n,w}]_{i}-[\mathbf{{x}}^{n,w}]_{i}\|^{2}\end{array}\nonumber \\
 & \geq & \begin{array}{c}
f_{i,n,w-1}([\mathbf{{x}}^{n,w-1}]_{i})+\frac{1}{2}\|[\bar{\mathbf{{x}}}-\mathbf{{v}}_{H}^{n,w}]_{i}-[\mathbf{{x}}^{n,w-1}]_{i}\|^{2}+\frac{1}{2}\|[\mathbf{{x}}^{n,w-1}]_{i}-[\mathbf{{x}}^{n,w}]_{i}\|^{2}.\end{array}\nonumber 
\end{eqnarray}
The equation in \eqref{eq:for-beta-chain} comes from the fact that
$[\mathbf{{x}}^{n,w}]_{i}$ being the minimizer of \eqref{eq:alg-primal-subpblm}
is such that $f_{i,n,w-1}([\mathbf{{x}}^{n,w}]_{i})=\tilde{f}_{i,n,w-1}([\mathbf{{x}}^{n,w}]_{i})$,
and $f_{i,n,w}(\cdot)$ is designed through \eqref{eq:finw-design}
so that $f_{i,n,w}([\mathbf{{x}}^{n,w}]_{i})=f_{i,n,w-1}([\mathbf{{x}}^{n,w}]_{i})$.
The inequality in \eqref{eq:for-beta-chain} follows from the design
of $f_{i,n,w-1}(\cdot)$ through \eqref{eq:finw-design}, which implies
that $[\mathbf{{x}}^{n,w-1}]_{i}$ is the minimizer of $f_{i,n,w-1}(\cdot)+\frac{1}{2}\|[\bar{\mathbf{{x}}}-\mathbf{{v}}_{H}^{n,w}]_{i}-\cdot\|^{2}$.

Since $S_{n,w}\cap\bar{\mathcal{E}}=\emptyset$, we have $\mathbf{{v}}_{H}^{n,w-1}=\mathbf{{v}}_{H}^{n,w}$.
Let $\beta_{i}$ be defined by 
\begin{eqnarray}
\beta_{i} & := & \begin{array}{c}
\left(f_{i,n,w}^{*}([\mathbf{{z}}_{i}^{n,w}]_{i})+\frac{1}{2}\|[\bar{\mathbf{{x}}}-\mathbf{{v}}_{H}^{n,w}]_{i}-[\mathbf{{z}}_{i}^{n,w}]_{i}\|^{2}\right)\end{array}\label{eq:beta-form}\\
 &  & \begin{array}{c}
-\big(f_{i,n,w-1}^{*}([\mathbf{{z}}_{i}^{n,w-1}]_{i})+\frac{1}{2}\|[\bar{\mathbf{{x}}}-\mathbf{{v}}_{H}^{n,w-1}]_{i}-[\mathbf{{z}}_{i}^{n,w-1}]_{i}\|^{2}\big).\end{array}\nonumber 
\end{eqnarray}
Proposition \ref{prop:P-D-of-one-block} implies that 
\begin{eqnarray}
 &  & \begin{array}{c}
f_{i,n,w}^{*}([\mathbf{{z}}_{i}^{n,w}]_{i})+\frac{1}{2}\|[\bar{\mathbf{{x}}}-\mathbf{{v}}_{H}^{n,w}]_{i}-[\mathbf{{z}}_{i}^{n,w}]_{i}\|^{2}\end{array}\label{eq:use-prop}\\
 & = & \begin{array}{c}
-f_{i,n,w}([\mathbf{{x}}^{n,w}]_{i})+\frac{1}{2}\|[\bar{\mathbf{{x}}}-\mathbf{{v}}_{H}^{n,w}]_{i}\|^{2}-\frac{1}{2}\|[\bar{\mathbf{{x}}}-\mathbf{{v}}_{H}^{n,w}]_{i}-[\mathbf{{x}}^{n,w}]_{i}\|^{2}.\end{array}\nonumber 
\end{eqnarray}
An equation similar to \eqref{eq:use-prop} involving $f_{i,n,w-1}(\cdot)$
plugged into \eqref{eq:for-beta-chain} and \eqref{eq:beta-form},
and the fact that $\mathbf{{v}}_{H}^{n,w}=\mathbf{{v}}_{H}^{n,w-1}$
gives $\beta_{i}\geq\frac{1}{2}\|[\mathbf{{z}}_{i}^{n,w-1}-\mathbf{{z}}_{i}^{n,w}]_{i}\|^{2}$.
One can easily check from the definitions that 
\[
\begin{array}{c}
F_{n,w}(\{\mathbf{{z}}_{\alpha}^{n,w}\}_{\alpha\in\bar{\mathcal{E}}\cup\mathcal{V}})-F_{n,w-1}(\{\mathbf{{z}}_{\alpha}^{n,w-1}\}_{\alpha\in\bar{\mathcal{E}}\cup\mathcal{V}})=\underset{i\in S_{n,w}}{\sum}\beta_{i},\end{array}
\]
which leads to our result. The proof of the second statement is exactly
the same. 
\end{proof}
We remark on the design of Algorithm \ref{alg:Ext-Dyk}.
\begin{rem}
(On improving the affine models) In our design of Algorithm \ref{alg:Ext-Dyk},
we improve the affine model $f_{i,n,w}(\cdot)$ for $i\in\mathcal{V}_{4}$
only if $S_{n,w}\subset\mathcal{V}_{4}$. It is easy to see that we
can apply the observation in Remark \ref{rem:min-max-of-2-quads}
to minimize the maximum of two quadratics \eqref{eq:alg-primal-subpblm}
analytically, but doing so without Assumption \ref{assu:to-start-subalg}
would affect the convergence proof. 
\end{rem}

\subsection{Further new steps in convergence proof}

Since the proof of convergence shares many similarities to the original
proof in \cite{Pang_Dist_Dyk}, we describe the new steps of the proof
in this subsection that were not already covered and defer the rest
of the proof to the appendix.  

Recall the definition of $\mathbf{f}_{\alpha,n,w}(\cdot)$ in \eqref{eq_m:h-a-n-w}.
We have the following easy claim.
\begin{claim}
\label{claim:Fenchel-duality}In Algorithm \ref{alg:Ext-Dyk}, for
all $\alpha\in S_{n,w}$, we have 

\begin{enumerate}
\item [(a)]$-\mathbf{{x}}^{n,w}+\partial\mathbf{f}_{\alpha,n,w}^{*}(\mathbf{{z}}_{\alpha}^{n,w})\ni0$,
\item [(b)]$-\mathbf{{z}}_{\alpha}^{n,w}+\partial\mathbf{f}_{\alpha,n,w}(\mathbf{{x}}^{n,w})\ni0$,
and
\item [(c)]$\mathbf{f}_{\alpha,n,w}(\mathbf{{x}}^{n,w})+\mathbf{f}_{\alpha,n,w}^{*}(\mathbf{{z}}_{\alpha}^{n,w})=\langle\mathbf{{x}}^{n,w},\mathbf{{z}}_{\alpha}^{n,w}\rangle$. 
\end{enumerate}
\end{claim}

\begin{proof}
There are two cases. The first case is when \eqref{eq:Dykstra-min-subpblm}
is invoked. By taking the optimality conditions in \eqref{eq:Dykstra-min-subpblm}
with respect to $z_{\alpha}$ for $\alpha\in S_{n,w}$ and making
use of \eqref{eq_m:all_acronyms} to get $\mathbf{{x}}^{n,w}=\bar{\mathbf{{x}}}-\sum_{\alpha\in\bar{\mathcal{E}}\cup\mathcal{V}}\mathbf{{z}}_{\alpha}^{n,w}$,
we deduce (a). The second case is when Algorithm \ref{alg:subdiff-subalg}
is invoked, and is similar. The equivalence of (a), (b) and (c) is
standard. 
\end{proof}
For all valid $(n,w)$, since $\mathbf{f}_{\alpha,n,w}(\cdot)\leq\mathbf{f}_{\alpha}(\cdot)$
for all $\alpha\in\mathcal{V}_{4}$, we have $\mathbf{f}_{\alpha,n,w}^{*}(\cdot)\geq\mathbf{f}_{\alpha}^{*}(\cdot)$.
Let $D_{\alpha,n}$ and $E_{\alpha,n}$ be defined to be \begin{subequations}\label{eq_m:error-def}
\begin{eqnarray}
D_{\alpha,n} & := & \mathbf{f}_{\alpha,n,p(n,\alpha)}^{*}(\mathbf{{z}}_{\alpha}^{n,p(n,\alpha)})-\mathbf{f}_{\alpha}^{*}(\mathbf{{z}}_{\alpha}^{n,p(n,\alpha)})\geq0\label{eq:D-error-def}\\
\mbox{and }E_{\alpha,n} & := & \mathbf{f}_{\alpha}(\bar{x}-v_{A}^{n,p(n,\alpha)})-\mathbf{f}_{\alpha,n,p(n,\alpha)}(\bar{\mathbf{{x}}}-\mathbf{{v}}_{A}^{n,p(n,\alpha)})\geq0.\label{eq:E-error-def}
\end{eqnarray}
\end{subequations}When $\alpha\in[\mathcal{\bar{E}\cup}\mathcal{V}]\backslash\mathcal{V}_{4}$,
then $E_{\alpha,n}=D_{\alpha,n}=0$ for all $n$. Next, we have 
\begin{eqnarray}
 &  & \mathbf{f}_{\alpha}^{*}(\mathbf{{z}}_{\alpha}^{n,p(n,\alpha)})+\mathbf{f}_{\alpha}(\bar{\mathbf{{x}}}-\mathbf{{v}}_{A}^{n,p(n,\alpha)})\label{eq:error-deriv}\\
 & \overset{\eqref{eq_m:error-def}}{=} & \mathbf{f}_{\alpha,n,p(n,\alpha)}^{*}(\mathbf{{z}}_{\alpha}^{n,p(n,\alpha)})+\mathbf{f}_{\alpha,n,p(n,\alpha)}(\bar{\mathbf{{x}}}-\mathbf{{v}}_{A}^{n,p(n,\alpha)})+E_{\alpha,n}-D_{\alpha,n}\nonumber \\
 & \overset{\scriptsize{\alpha\in S_{n,p(n,\alpha)},\mbox{ Claim \ref{claim:Fenchel-duality}}}}{=} & \langle\mathbf{{z}}_{\alpha}^{n,p(n,\alpha)},\bar{\mathbf{{x}}}-\mathbf{{v}}_{A}^{n,p(n,\alpha)}\rangle+E_{\alpha,n}-D_{\alpha,n}.\nonumber 
\end{eqnarray}

We now state the main convergence theorem of this paper. 
\begin{thm}
\label{thm:convergence} (Convergence to primal minimizer) Consider
Algorithm \ref{alg:Ext-Dyk}. Assume that the problem \eqref{eq:Dyk-primal}
is feasible, and for all $n\geq1$, $\bar{\mathcal{E}}_{n}=[\cup_{w=1}^{\bar{w}}S_{n,w}]\cap\bar{\mathcal{E}}$,
and $[\cup_{w=1}^{\bar{w}}S_{n,w}]\supset\mathcal{V}$. Suppose that
Assumption \ref{assu:to-start-subalg} holds.

For the sequence $\{\mathbf{{z}}_{\alpha}^{n,w}\}_{{1\leq n<\infty\atop 0\leq w\leq\bar{w}}}\subset[\mathbb{R}^{m}]^{|\mathcal{V}|}$
for each $\alpha\in\bar{\mathcal{E}}\cup\mathcal{V}$ generated by
Algorithm \ref{alg:Ext-Dyk} and the sequences $\{\mathbf{{v}}_{H}^{n,w}\}_{{1\leq n<\infty\atop 0\leq w\leq\bar{w}}}\subset[\mathbb{R}^{m}]^{|\mathcal{V}|}$
and $\{\mathbf{{v}}_{A}^{n,w}\}_{{1\leq n<\infty\atop 0\leq w\leq\bar{w}}}\subset[\mathbb{R}^{m}]^{|\mathcal{V}|}$
thus derived, we have:

\begin{enumerate}
\item [(i)]The sum $\sum_{n=1}^{\infty}\sum_{w=1}^{\bar{w}}\|\mathbf{{v}}_{A}^{n,w}-\mathbf{{v}}_{A}^{n,w-1}\|^{2}$
is finite and $\{F_{n,\bar{w}}(\{\mathbf{{z}}_{\alpha}^{n,\bar{w}}\}_{\alpha\in\bar{\mathcal{E}}\cup\mathcal{V}})\}_{n=1}^{\infty}$
is nondecreasing.
\item [(ii)]There is a constant $C$ such that $\|\mathbf{{v}}_{A}^{n,w}\|^{2}\leq C$
for all $n\in\mathbb{N}$ and $w\in\{1,\dots,\bar{w}\}$. 
\item [(iii)]For all $i\in\mathcal{V}_{3}\cup\mathcal{V}_{4}$, $n\geq1$
and $w\in\{1,\dots,\bar{w}\}$, the vectors $\mathbf{{z}}_{i}^{n,w}$
are bounded. 
\end{enumerate}
Suppose also that

\begin{enumerate}
\item [(1)]There are constants $A$ and $B$ such that 
\begin{equation}
\sum_{\alpha\in\mathcal{\bar{E}}\cup\mathcal{V}}\|\mathbf{{z}}_{\alpha}^{n,\bar{w}}\|\leq A\sqrt{n}+B\mbox{ for all }n\geq0.\label{eq:sqrt-growth-sum-z}
\end{equation}
\end{enumerate}
Then

\begin{enumerate}
\item [(iv)]For all $\alpha\in[\mathcal{\bar{E}}\cup\mathcal{V}]\backslash\mathcal{V}_{4}$,
we have $E_{\alpha,n}=0$. Also, for all $i\in\mathcal{V}_{4}$, we
have $\lim_{n\to\infty}E_{i,n}=0$. 
\item [(v)]There exists a subsequence $\{\mathbf{{v}}_{A}^{n_{k},\bar{w}}\}_{k=1}^{\infty}$
of $\{\mathbf{{v}}_{A}^{n,\bar{w}}\}_{n=1}^{\infty}$ which converges
to some $\mathbf{{v}}_{A}^{*}\in[\mathbb{R}^{m}]^{|\mathcal{V}|}$
and that 
\[
\lim_{k\to\infty}\langle\mathbf{{v}}_{A}^{n_{k},\bar{w}}-\mathbf{{v}}_{A}^{n_{k},p(n_{k},\alpha)},\mathbf{{z}}_{\alpha}^{n_{k},\bar{w}}\rangle=0\mbox{ for all }\alpha\in\bar{\mathcal{E}}\cup\mathcal{V}.
\]
\item [(vi)]Let $\mathbf{f}(\cdot)=\sum_{\alpha\in\bar{\mathcal{E}}\cup\mathcal{V}}\mathbf{f}_{\alpha}(\cdot)$.
For the $\mathbf{{v}}_{A}^{*}$ in (v), $\mathbf{{x}}_{0}-\mathbf{{v}}_{A}^{*}$
is the minimizer of the primal problem \eqref{eq:Dyk-primal} and
\begin{equation}
\lim_{k\to\infty}F_{n_{k},w}(\{\mathbf{{z}}_{\alpha}^{n_{k},w}\}_{\alpha\in\bar{\mathcal{E}}\cup\mathcal{V}})=\lim_{k\to\infty}F(\{\mathbf{{z}}_{\alpha}^{n_{k},w}\}_{\alpha\in\bar{\mathcal{E}}\cup\mathcal{V}})=\frac{1}{2}\|\mathbf{{v}}_{A}^{*}\|^{2}+\mathbf{f}(\bar{\mathbf{{x}}}-\mathbf{{v}}_{A}^{*}).\label{eq:thm-iv-concl}
\end{equation}
\end{enumerate}
The properties (i) to (vi) in turn imply that $\lim_{n\to\infty}\mathbf{{x}}^{n,\bar{w}}$
exists and equals $\bar{\mathbf{{x}}}-\mathbf{{v}}_{A}^{*}$, which
is the primal minimizer of \eqref{eq:Dyk-primal}.
\end{thm}

The proofs of parts (i), (ii), (v) and (vi) are similar to the proof
in \cite{Pang_Dist_Dyk}, and (iii) and (iv) are new. We shall prove
(iii) and (iv) here and defer the rest of the proof to the appendix.
\begin{proof}
[Proof of Theorem \ref{thm:convergence}(iii)]In view of line 14 in
Algorithm \ref{alg:Ext-Dyk}, it suffices to prove that $\mathbf{{z}}_{i}^{n,w}$
is bounded if $i\in S_{n,w}$. By the sparsity pattern in Proposition
\ref{prop:sparsity}, for each $i\in\mathcal{V}_{3}\cup\mathcal{V}_{4}$,
$\mathbf{{z}}_{i}^{n,w}$ is bounded if and only if $[\mathbf{{z}}_{i}^{n,w}]_{i}$
is bounded. Since $\{[\bar{\mathbf{{x}}}-\mathbf{{v}}_{A}^{n,w}]_{i}\}_{{1\leq n<\infty\atop 0\leq w\leq\bar{w}}}$
is bounded by (ii), it is clear that $\{[\mathbf{{z}}_{i}^{n,w}]_{i}\}_{{1\leq n<\infty\atop 0\leq w\leq\bar{w}}}$
is bounded if and only if $\{[\bar{\mathbf{{x}}}-\mathbf{{v}}_{H}^{n,w}]_{i}\}_{{1\leq n<\infty\atop 0\leq w\leq\bar{w}}}$
is bounded. Seeking a contradiction, suppose $\{[\bar{\mathbf{{x}}}-\mathbf{{v}}_{H}^{n,w}]_{i}\}_{{1\leq n<\infty\atop 0\leq w\leq\bar{w}}}$
is unbounded. We look at the problem 
\begin{equation}
\begin{array}{c}
\underset{x\in\mathbb{R}^{m}}{\min}\frac{1}{2}\|[\bar{\mathbf{{x}}}-\mathbf{{v}}_{H}^{n,w}]_{i}-x\|^{2}+f_{i}(x)\end{array}\label{eq:in-pf-primal-1}
\end{equation}
and consider two possibilities. Let $\tilde{x}_{i}^{n,w}$ be the
primal solution to \eqref{eq:in-pf-primal-1}. Note that if $i\in\mathcal{V}_{3}$,
then $\tilde{x}_{i}^{n,w}$ is $[\mathbf{{x}}^{n,w}]_{i}$. If the
$\{\tilde{x}_{i}^{n,w}\}_{n,w}$ are bounded, then the dual solution
of \eqref{eq:in-pf-primal-1} is $[\bar{\mathbf{{x}}}-\mathbf{{v}}_{H}^{n,w}]_{i}-\tilde{x}_{i}^{n,w}$,
which will be unbounded. A standard compactness argument shows that
there is a point $\tilde{x}\in\mathbb{R}^{m}$ for which the set $\partial f_{i}(\tilde{x})$
is unbounded, which contradicts $\dom(f_{i})=\mathbb{R}^{m}$. 

If the corresponding primal solutions $\tilde{x}_{i}^{n,w}$ are unbounded,
consider $\{\tilde{\mathbf{{z}}}_{\alpha}^{n,w}\}_{\alpha\in\bar{\mathcal{E}}\cup\mathcal{V}}$,
where 
\[
\tilde{\mathbf{{z}}}_{\alpha}^{n,w}=\mathbf{{z}}_{\alpha}^{n,w}\mbox{ if }\alpha\in[\bar{\mathcal{E}}\cup\mathcal{V}]\backslash\{i\}\mbox{ and }[\tilde{\mathbf{{z}}}_{i}^{n,w}]_{j}=\begin{cases}
[\bar{\mathbf{{x}}}-\mathbf{{v}}_{H}^{n,w}]_{i}-\tilde{x}_{i}^{n,w} & \mbox{ if }j=i\\
0 & \mbox{ otherwise. }
\end{cases}
\]
Let $\tilde{F}_{n,w}(\cdot)$ be defined to be
\begin{equation}
\begin{array}{c}
\tilde{F}_{n,w}(\{\mathbf{{z}}_{\alpha}\}_{\alpha\in\bar{\mathcal{E}}\cup\mathcal{V}}):=-\frac{1}{2}\left\Vert \bar{\mathbf{{x}}}-\underset{\alpha\in\bar{\mathcal{E}}\cup\mathcal{V}}{\sum}\mathbf{{z}}_{\alpha}\right\Vert ^{2}+\frac{1}{2}\|\bar{\mathbf{{x}}}\|^{2}-\underset{\alpha\in[\bar{\mathcal{E}}\cup\mathcal{V}]\backslash\{i\}}{\sum}\mathbf{f}_{\alpha,n,w}(\mathbf{{z}}_{\alpha})-\mathbf{f}_{i}(\mathbf{{z}}_{i})\end{array}\label{eq:Dykstra-dual-defn-2}
\end{equation}
 Then $F_{n,w}(\cdot)\leq\tilde{F}_{n,w}(\cdot)\leq F(\cdot)$. Also,
Proposition \ref{prop:P-D-of-one-block} shows that $[\tilde{\mathbf{{z}}}_{i}^{n,w}]_{i}$
is the dual solution to \eqref{eq:in-pf-primal-1}. So 
\begin{equation}
F_{n,w}(\{\mathbf{{z}}_{\alpha}^{n,w}\}_{\alpha\in\bar{\mathcal{E}}\cup\mathcal{V}})\leq\tilde{F}_{n,w}(\{\mathbf{{z}}_{\alpha}^{n,w}\}_{\alpha\in\bar{\mathcal{E}}\cup\mathcal{V}})\leq\tilde{F}_{n,w}(\{\tilde{\mathbf{{z}}}_{\alpha}^{n,w}\}_{\alpha\in\bar{\mathcal{E}}\cup\mathcal{V}}\})\leq F(\{\tilde{\mathbf{{z}}}_{\alpha}^{n,w}\}_{\alpha\in\bar{\mathcal{E}}\cup\mathcal{V}}\}).\label{eq:big-F-chain}
\end{equation}
Next, suppose $\mathbf{{x}}^{*}$ is a solution of \eqref{eq:Dyk-primal}.
Then 
\begin{eqnarray*}
 &  & \begin{array}{c}
\frac{1}{2}\|\bar{\mathbf{{x}}}-\mathbf{{x}}^{*}\|^{2}+\underset{\alpha\in\bar{\mathcal{E}}\cup\mathcal{V}}{\sum}\mathbf{f}_{\alpha}(\mathbf{{x}}^{*})-F_{n,w}(\{\mathbf{{z}}_{\alpha}^{n,w}\}_{\alpha\in\bar{\mathcal{E}}\cup\mathcal{V}})\end{array}\\
 & \overset{\eqref{eq:big-F-chain}}{\geq} & \begin{array}{c}
\frac{1}{2}\|\bar{\mathbf{{x}}}-\mathbf{{x}}^{*}\|^{2}+\underset{\alpha\in\bar{\mathcal{E}}\cup\mathcal{V}}{\sum}\mathbf{f}_{\alpha}(\mathbf{{x}}^{*})-F(\{\tilde{\mathbf{{z}}}_{\alpha}^{n,w}\}_{\alpha\in\bar{\mathcal{E}}\cup\mathcal{V}})\end{array}\\
 & \overset{\eqref{eq:From-8}}{\geq} & \begin{array}{c}
\frac{1}{2}\left\Vert \bar{\mathbf{{x}}}-\mathbf{{x}}^{*}-\underset{\alpha\in\bar{\mathcal{E}}\cup\mathcal{V}}{\sum}\tilde{\mathbf{{z}}}_{\alpha}^{n,w}\right\Vert ^{2}\end{array}\\
 & \overset{\scriptsize{\mbox{Take }i\mbox{-th component only}}}{\geq} & \begin{array}{c}
\frac{1}{2}\left\Vert \tilde{x}_{i}^{n,w}-[\mathbf{{x}}^{*}]_{i}\right\Vert ^{2}.\end{array}
\end{eqnarray*}
The above inequality and the unboundedness of $\tilde{x}_{i}^{n,w}$
implies that the duality gap would go to infinity, which contradicts
part (i). Thus we are done. 
\end{proof}

\begin{proof}
[Proof of Theorem \ref{thm:convergence}(iv)]The first sentence of
this claim is immediate from \eqref{eq:h-a-n-w-eq-h-a}. We now prove
the second sentence. Seeking a contradiction, suppose that $\limsup_{n\to\infty}E_{i,n}>0$.
In Algorithm \ref{alg:subdiff-subalg}, in view of Assumption \ref{assu:to-start-subalg}(2),
$[\mathbf{{z}}_{i}^{n,\bar{w}}]_{i}$ is the minimizer of the problem
\begin{equation}
\begin{array}{c}
\underset{z\in\mathbb{R}^{m}}{\min}\frac{1}{2}\|[\bar{\mathbf{{x}}}-\mathbf{{v}}_{H}^{n,\bar{w}}]_{i}-z\|^{2}+f_{i,n,\bar{w}}^{*}(z).\end{array}\label{eq:in-pf-dual}
\end{equation}
The associated primal problem is, up to a constant independent of
$x$, 
\begin{equation}
\begin{array}{c}
\underset{x\in\mathbb{R}^{m}}{\min}\frac{1}{2}\|[\bar{\mathbf{{x}}}-\mathbf{{v}}_{H}^{n,\bar{w}}]_{i}-x\|^{2}+f_{i,n,\bar{w}}(x).\end{array}\label{eq:in-pf-primal}
\end{equation}
The primal solution is $[\bar{\mathbf{{x}}}-\mathbf{{v}}_{H}^{n,\bar{w}}]_{i}-[\mathbf{{z}}_{i}^{n,\bar{w}}]_{i}\overset{\eqref{eq_m:all_acronyms}}{=}[\bar{\mathbf{{x}}}-\mathbf{{v}}_{A}^{n,\bar{w}}]_{i}$.
The dual solution is $[\mathbf{{z}}_{i}^{n,\bar{w}}]_{i}$. So 
\begin{equation}
[\mathbf{{x}}_{i}^{n,\bar{w}}]_{i}\in\partial f_{i,n,\bar{w}}([\bar{\mathbf{{x}}}-\mathbf{{v}}_{A}^{n,\bar{w}}]_{i}).\label{eq:subgrad-birth}
\end{equation}
Recall Assumption \ref{assu:to-start-subalg}(3) and $\mathbf{{v}}_{H}^{n,\bar{w}}=\mathbf{{v}}_{H}^{n+1,0}$
by line 16 in Algorithm \ref{alg:Ext-Dyk}. We now analyze the increase
in the dual objective value of each separate problem: 
\begin{eqnarray*}
\Delta_{i,n} & := & \begin{array}{c}
\big[f_{i,n+1,0}^{*}([\mathbf{{z}}_{i}^{n+1,0}]_{i})+\frac{1}{2}\|[\bar{\mathbf{{x}}}-\mathbf{{v}}_{H}^{n+1,0}]_{i}-[\mathbf{{z}}_{i}^{n+1,0}]_{i}\|^{2}\big]\end{array}\\
 &  & \begin{array}{c}
-\big[f_{i,n+1,1}^{*}([\mathbf{{z}}_{i}^{n+1,1}]_{i})+\frac{1}{2}\|[\bar{\mathbf{{x}}}-\mathbf{{v}}_{H}^{n+1,1}]_{i}-[\mathbf{{z}}_{i}^{n+1,1}]_{i}\|^{2}\big].\end{array}
\end{eqnarray*}
Recall that $E_{i,n}$ is also $f_{i}([\mathbf{{x}}^{n,\bar{w}}]_{i})-f_{i,n,\bar{w}}([\mathbf{{x}}^{n,\bar{w}}]_{i})=f_{i}([\mathbf{{x}}^{n+1,0}]_{i})-f_{i,n+1,0}([\mathbf{{x}}^{n+1,0}]_{i}).$
Proposition \ref{prop:P-D-of-one-block} and Assumption \ref{assu:to-start-subalg}(2)
tell us that 
\begin{eqnarray*}
 &  & \begin{array}{c}
f_{i,n+1,0}^{*}([\mathbf{{z}}_{i}^{n+1,0}]_{i})+\frac{1}{2}\|[\bar{\mathbf{{x}}}-\mathbf{{v}}_{H}^{n+1,0}]_{i}-[\mathbf{{z}}_{i}^{n+1,0}]_{i}\|^{2}\big]\end{array}\\
 & = & \begin{array}{c}
-f_{i,n+1,0}([\mathbf{{x}}^{n+1,0}]_{i})+\frac{1}{2}\|[\bar{\mathbf{{x}}}-\mathbf{{v}}_{H}^{n+1,0}]_{i}\|^{2}-\frac{1}{2}\|[\bar{\mathbf{{x}}}-\mathbf{{v}}_{H}^{n+1,0}]_{i}-[\mathbf{{x}}^{n+1,0}]_{i}\|^{2}\big].\end{array}
\end{eqnarray*}
A similar result holds for the problem involving $f_{i,n+1,1}(\cdot)$.
Since $S_{n+1,1}\cap\bar{\mathcal{E}}=\emptyset$ and $\mathbf{{v}}_{H}^{n+1,0}=\mathbf{{v}}_{H}^{n+1,1}$,
we have 
\begin{eqnarray*}
\Delta_{i,n} & = & \begin{array}{c}
\big[f_{i,n+1,1}([\mathbf{{x}}^{n+1,1}]_{i})+\frac{1}{2}\|[\bar{\mathbf{{x}}}-\mathbf{{v}}_{H}^{n+1,1}]_{i}-[\mathbf{{x}}^{n+1,1}]_{i}\|^{2}\big]\end{array}\\
 &  & \begin{array}{c}
-\big[f_{i,n+1,0}([\mathbf{{x}}^{n+1,0}]_{i})+\frac{1}{2}\|[\bar{\mathbf{{x}}}-\mathbf{{v}}_{H}^{n+1,0}]_{i}-[\mathbf{{x}}^{n+1,0}]_{i}\|^{2}\big].\end{array}
\end{eqnarray*}
The analogue of Lemma \ref{lem:alpha-recurrs}(1) tells us that $\Delta_{i,n}\geq\frac{1}{2}t_{i,n}^{2}$,
where $t_{i,n}$ is the positive root satisfying 
\[
\begin{array}{c}
\frac{1}{2}t_{i,n}^{2}+\|s_{i,n,\bar{w}}+[\mathbf{{x}}^{n+1,0}]_{i}-[\bar{\mathbf{{x}}}-\mathbf{{v}}_{H}^{n+1,0}]_{i}\|t_{i,n}=E_{i,n},\end{array}
\]
where $s_{i,n,\bar{w}}\in\partial f_{i}([\mathbf{{x}}^{n,\bar{w}}]_{i})$
is the subgradient used to form the linearization of $f(\cdot)$ at
$[\mathbf{{x}}^{n+1,0}]_{i}$. Note that $[\mathbf{{x}}^{n,\bar{w}}]_{i}-[\bar{\mathbf{{x}}}-\mathbf{{v}}_{H}^{n,\bar{w}}]_{i}=-[\mathbf{{z}}_{i}^{n,\bar{w}}]_{i}$,
so the term $\|s_{i,n,\bar{w}}+[\mathbf{{x}}^{n,\bar{w}}]_{i}-[\bar{\mathbf{{x}}}-\mathbf{{v}}_{H}^{n,\bar{w}}]_{i}\|$
becomes $\|s_{i,n,\bar{w}}-[\mathbf{{z}}_{i}^{n,\bar{w}}]_{i}\|$.
Since both $s_{i,n,\bar{w}}$ and $[\mathbf{{z}}_{i}^{n,\bar{w}}]_{i}$
are bounded, $\limsup_{n\to\infty}t_{i,n}>0$, and so $\limsup_{n\to\infty}\Delta_{i,n}>0$.
We can check from the definitions that 
\[
\begin{array}{c}
F_{n+1,1}(\{\mathbf{{z}}_{\alpha}^{n+1,1}\}_{\alpha\in\bar{\mathcal{E}}\cup\mathcal{V}})-F_{n+1,0}(\{\mathbf{{z}}_{\alpha}^{n+1,0}\}_{\alpha\in\bar{\mathcal{E}}\cup\mathcal{V}})=\underset{i\in S_{n,\bar{w}}}{\sum}\Delta_{i,n}.\end{array}
\]
This means that the dual objective value can increase indefinitely,
which then implies that the problem \eqref{eq:distrib-dyk-primal}
is infeasible, which is a contradiction. 
\end{proof}
Proposition \ref{prop:control-growth} below shows some reasonable
conditions to guarantee \eqref{eq:sqrt-growth-sum-z}. The ideas of
its proof were already present in \cite{Pang_Dyk_spl,Pang_Dist_Dyk},
so we defer its proof to the appendix.
\begin{prop}
\label{prop:control-growth}(Growth of $\sum_{\alpha\in\bar{\mathcal{E}}\cup\mathcal{V}}\|\mathbf{{z}}_{\alpha}^{n,w}\|$)
In Algorithm \ref{alg:Ext-Dyk}, suppose:

\begin{enumerate}
\item There are only finitely many $S_{n,w}$ for which $S_{n,w}\cap[\mathcal{V}_{1}\cup\mathcal{V}_{2}]$
contains more than one element.
\item There are constants $M_{1}>0$ and $M_{2}>0$ such that the size of
the set 
\[
\big\{(n',w):n'\leq n,\,w\in\{1,\dots,\bar{w}\},\,|S_{n',w}\cap\mathcal{V}|>1\big\}
\]
is bounded by $M_{1}\sqrt{n}+M_{2}$ for all $n$. 
\end{enumerate}
Then condition (1) in Theorem \ref{thm:convergence} holds.
\end{prop}

\subsection{\label{subsec:composition-lin-op}Composition with a linear operator }

Suppose some $f_{i}(\cdot)$ were defined as $f_{i,1}\circ A_{i}(\cdot)$,
where $f:Y\to\mathbb{R}$ is a closed convex function, $Y$ is another
finite dimensional Hilbert space and $A_{i}:\mathbb{R}^{m}\to Y$
is a linear map. One may either still try to take the proximal mapping
of $f_{i}(\cdot)$, but it may involve some expensive operations on
$A_{i}$. Alternatively, we can write or we can write $f_{i,1}\circ A_{i}(x_{i})$
as 
\[
f_{i}(y_{i})+\delta_{\{(x,y):A_{i}x=y\}}(x_{i},y_{i}),
\]
which splits into the sum of two functions. Note however that since
we require the problem to be strongly convex, creating the new variable
$y$ adds new regularizing terms to the objective function. 

\section{Conclusion}

The main contribution in this paper is to show that the distributed
Dykstra's algorithm can be extended to incorporate subdifferentiable
functions in a natural manner so that the algorithm converges to the
primal minimizer, even if there is no dual minimizer. A next question
is to find convergence rates of the algorithm. The derivation of such
rates uses rather different techniques from that of this paper, and
requires additional conditions to ensure the existence of a dual minimizer.
We defer this to \cite{Pang_rate_D_Dyk}, where we also perform numerical
experiments that show that the distributed Dykstra's algorithm is
sound. 

\appendix

\section{Further proofs}

In this appendix, we completing the parts of the proofs of Theorem
\ref{thm:convergence} and \ref{prop:control-growth} that we consider
to be too similar to the ones in \cite{Pang_Dist_Dyk}.

The following inequality describes the duality gap between \eqref{eq:Dyk-primal}
and \eqref{eq:dual-fn}. 
\begin{eqnarray}
 &  & \begin{array}{c}
\frac{1}{2}\|\bar{\mathbf{{x}}}-\mathbf{{x}}\|^{2}+\underset{\alpha\in\bar{\mathcal{E}}\cup\mathcal{V}}{\sum}\mathbf{f}_{\alpha}(\mathbf{{x}})-F(\{\mathbf{{z}}_{\alpha}\}_{\alpha\in\bar{\mathcal{E}}\cup\mathcal{V}})\end{array}\label{eq:From-8}\\
 & \overset{\eqref{eq:Dykstra-dual-defn}}{=} & \begin{array}{c}
\frac{1}{2}\|\bar{\mathbf{{x}}}-\mathbf{{x}}\|^{2}+\underset{\alpha\in\bar{\mathcal{E}}\cup\mathcal{V}}{\sum}[\mathbf{f}_{\alpha}(\mathbf{{x}})+\mathbf{f}_{\alpha}^{*}(\mathbf{{z}}_{\alpha})]-\left\langle \bar{\mathbf{{x}}},\underset{\alpha\in\bar{\mathcal{E}}\cup\mathcal{V}}{\sum}\mathbf{{z}}_{\alpha}\right\rangle +\frac{1}{2}\left\Vert \underset{\alpha\in\bar{\mathcal{E}}\cup\mathcal{V}}{\sum}\mathbf{{z}}_{\alpha}\right\Vert ^{2}\end{array}\nonumber \\
 & \overset{\scriptsize\mbox{Fenchel duality}}{\geq} & \begin{array}{c}
\frac{1}{2}\|\bar{\mathbf{{x}}}-\mathbf{{x}}\|^{2}+\left\langle \mathbf{{x}},\underset{\alpha\in\bar{\mathcal{E}}\cup\mathcal{V}}{\sum}\mathbf{{z}}_{\alpha}\right\rangle -\left\langle \bar{\mathbf{{x}}},\underset{\alpha\in\bar{\mathcal{E}}\cup\mathcal{V}}{\sum}\mathbf{{z}}_{\alpha}\right\rangle +\frac{1}{2}\left\Vert \underset{\alpha\in\bar{\mathcal{E}}\cup\mathcal{V}}{\sum}\mathbf{{z}}_{\alpha}\right\Vert ^{2}\end{array}\nonumber \\
 & = & \begin{array}{c}
\frac{1}{2}\left\Vert \bar{\mathbf{{x}}}-\mathbf{{x}}-\underset{\alpha\in\bar{\mathcal{E}}\cup\mathcal{V}}{\sum}\mathbf{{z}}_{\alpha}\right\Vert ^{2}\geq0.\end{array}\nonumber 
\end{eqnarray}

We continue with proving the rest of Theorem \ref{thm:convergence}.
\begin{proof}
[Proof of rest of Theorem \ref{thm:convergence}] We first show that
(i) to (vi) implies the final assertion. For all $n\in\mathbb{N}$
we have, from weak duality, 
\begin{equation}
\begin{array}{c}
F(\{\mathbf{{z}}_{\alpha}^{n,\bar{w}}\}_{\alpha\in\bar{\mathcal{E}}\cup\mathcal{V}})\leq\frac{1}{2}\|\bar{\mathbf{{x}}}-(\bar{\mathbf{{x}}}-\mathbf{{v}}_{A}^{*})\|^{2}+\underset{\alpha\in\bar{\mathcal{E}}\cup\mathcal{V}}{\overset{}{\sum}}\mathbf{f}_{\alpha}(\bar{\mathbf{{x}}}-\mathbf{{v}}_{A}^{*}).\end{array}\label{eq:weak-duality}
\end{equation}
Since the values $\{F_{n,w}(\{\mathbf{{z}}_{\alpha}^{n,\bar{w}}\}_{\alpha\in\bar{\mathcal{E}}\cup\mathcal{V}})\}_{n=1}^{\infty}$
are nondecreasing in $n$, we make use of (v) to get 
\[
\begin{array}{rcl}
\underset{n\to\infty}{\lim}F_{n,w}(\{\mathbf{{z}}_{\alpha}^{n,\bar{w}}\}_{\alpha\in\bar{\mathcal{E}}\cup\mathcal{V}}) & = & \frac{1}{2}\|\bar{\mathbf{{x}}}-(\bar{\mathbf{{x}}}-\mathbf{{v}}_{A}^{*})\|^{2}+\underset{\alpha\in\bar{\mathcal{E}}\cup\mathcal{V}}{\overset{}{\sum}}\mathbf{f}_{\alpha}(\bar{\mathbf{{x}}}-\mathbf{{v}}_{A}^{*}).\end{array}
\]
Since $F_{n,w}(\cdot)\leq F(\cdot)\leq\frac{1}{2}\|\bar{\mathbf{{x}}}-(\bar{\mathbf{{x}}}-\mathbf{{v}}_{A}^{*})\|^{2}+\sum_{\alpha\in\bar{\mathcal{E}}\cup\mathcal{V}}\mathbf{f}_{\alpha}(\bar{\mathbf{{x}}}-\mathbf{{v}}_{A}^{*})$,
we have \eqref{eq:thm-iv-concl}. Hence $\bar{\mathbf{{x}}}-\mathbf{{v}}_{A}^{*}=\arg\min_{\mathbf{{x}}}\mathbf{f}(\mathbf{{x}})+\frac{1}{2}\|\mathbf{{x}}-\bar{\mathbf{{x}}}\|^{2}$,
and (substituting $\mathbf{{x}}=\bar{\mathbf{{x}}}-\mathbf{{v}}_{A}^{*}$
in \eqref{eq:From-8}) 
\begin{eqnarray*}
 &  & \begin{array}{c}
\frac{1}{2}\|\bar{\mathbf{{x}}}-(\bar{\mathbf{{x}}}-\mathbf{{v}}_{A}^{*})\|^{2}+\mathbf{f}(\bar{\mathbf{{x}}}-\mathbf{{v}}_{A}^{*})-F(\{\mathbf{{z}}_{\alpha}^{n,\bar{w}}\}_{\alpha\in\mathcal{E}\cup\mathcal{V}})\end{array}\\
 & \overset{\eqref{eq:From-8},\eqref{eq:v-H-def},\eqref{eq:from-10}}{\geq} & \begin{array}{c}
\frac{1}{2}\|\bar{\mathbf{{x}}}-(\bar{\mathbf{{x}}}-\mathbf{{v}}_{A}^{*})-\mathbf{{v}}_{A}^{n,\bar{w}}\|^{2}\end{array}\\
 & \overset{\eqref{eq:x-from-v-A}}{=} & \begin{array}{c}
\frac{1}{2}\|\mathbf{{x}}^{n,\bar{w}}-(\bar{\mathbf{{x}}}-\mathbf{{v}}_{A}^{*})\|^{2}.\end{array}
\end{eqnarray*}
Hence $\lim_{n\to\infty}\mathbf{{x}}^{n,\bar{w}}$ is the minimizer
in (P). 

It remains to prove assertions (i) to (vi).

\textbf{Proof of (i):} We separate into two cases. 

We first consider the case when $S_{n,w}\not\subset\mathcal{V}_{4}$.
From the fact that $\{\mathbf{{z}}_{\alpha}^{n,w}\}_{\alpha\in S_{n,w}}$
minimizes  \eqref{eq:Dykstra-min-subpblm} (which includes the quadratic
regularizer) we have 
\begin{equation}
F_{n,w}(\{\mathbf{{z}}_{\alpha}^{n,w}\}_{\alpha\in\bar{\mathcal{E}}\cup\mathcal{V}})\overset{\eqref{eq:Dykstra-min-subpblm}}{\geq}\begin{array}{c}
F_{n,w-1}(\{\mathbf{{z}}_{\alpha}^{n,w-1}\}_{\alpha\in\bar{\mathcal{E}}\cup\mathcal{V}})+\frac{1}{2}\|\mathbf{{v}}_{A}^{n,w}-\mathbf{{v}}_{A}^{n,w-1}\|^{2}.\end{array}\label{eq:SHQP-decrease}
\end{equation}
(The last term in \eqref{eq:SHQP-decrease} arises from the quadratic
term in \eqref{eq:Dykstra-min-subpblm}.) By line 16 of Algorithm
\ref{alg:Ext-Dyk}, $\mathbf{{z}}_{i}^{n+1,0}=\mathbf{{z}}_{i}^{n,\bar{w}}$
for all $i\in\mathcal{V}$ and $\mathbf{{v}}_{H}^{n+1,0}=\mathbf{{v}}_{H}^{n,\bar{w}}$
(even though the decompositions \eqref{eq:reset-z-i-j-4} of $\mathbf{{v}}_{H}^{n+1,0}$
and $\mathbf{{v}}_{H}^{n,\bar{w}}$ may be different). 

In the second case when $S_{n,w}\subset\mathcal{V}_{4}$, Proposition
\ref{prop:quad-dec-case-2} and \eqref{eq_m:all_acronyms} show that
the inequality \eqref{eq:SHQP-decrease} holds.

Combining \eqref{eq:SHQP-decrease} over all $n'\in\{1,\dots,n\}$
and $w\in\{1,\dots,\bar{w}\}$, we have 
\[
\begin{array}{c}
F_{1,0}(\{\mathbf{{z}}_{\alpha}^{1,0}\}_{\alpha\in\bar{\mathcal{E}}\cup\mathcal{V}})+\underset{n'=1}{\overset{n}{\sum}}\underset{w=1}{\overset{\bar{w}}{\sum}}\|\mathbf{{v}}_{A}^{n',w}-\mathbf{{v}}_{A}^{n',w-1}\|^{2}\overset{\eqref{eq:SHQP-decrease}}{\leq}F_{n,\bar{w}}(\{\mathbf{{z}}_{\alpha}^{n,\bar{w}}\}_{\alpha\in\bar{\mathcal{E}}\cup\mathcal{V}}).\end{array}
\]
Next, $F_{n,\bar{w}}(\{\mathbf{{z}}_{\alpha}^{n,\bar{w}}\}_{\alpha\in\bar{\mathcal{E}}\cup\mathcal{V}})$
is bounded from above by weak duality. The proof of the claim is complete.

\textbf{Proof of (ii):} From part (i) and the fact that $F_{n,w}(\cdot)\leq F(\cdot)$,
we have 
\begin{equation}
-F_{1,0}(\{\mathbf{{z}}_{\alpha}^{1,0}\}_{\alpha\in\bar{\mathcal{E}}\cup\mathcal{V}})\overset{\scriptsize\mbox{part (i)}}{\geq}-F_{n,w}(\{\mathbf{{z}}_{\alpha}^{n,w}\}_{\alpha\in\bar{\mathcal{E}}\cup\mathcal{V}})\geq-F(\{\mathbf{{z}}_{\alpha}^{n,w}\}_{\alpha\in\bar{\mathcal{E}}\cup\mathcal{V}}).\label{eq:three-Fs}
\end{equation}
 Substituting $\mathbf{{x}}$ in \eqref{eq:From-8} to be the primal
minimizer $\mathbf{{x}}^{*}$ and $\{\mathbf{{z}}_{\alpha}\}_{\alpha\in\bar{\mathcal{E}}\cup\mathcal{V}}$
to be $\{\mathbf{{z}}_{\alpha}^{n,w}\}_{\alpha\in\bar{\mathcal{E}}\cup\mathcal{V}}$,
we have 
\begin{eqnarray*}
 &  & \begin{array}{c}
\frac{1}{2}\|\bar{\mathbf{{x}}}-\mathbf{{x}}^{*}\|^{2}+\underset{\alpha\in\bar{\mathcal{E}}\cup\mathcal{V}}{\overset{}{\sum}}\mathbf{f}_{\alpha}(\mathbf{{x}}^{*})-F_{1,0}(\{\mathbf{{z}}_{\alpha}^{1,0}\}_{\alpha\in\bar{\mathcal{E}}\cup\mathcal{V}})\end{array}\\
 & \overset{\eqref{eq:three-Fs}}{\geq} & \begin{array}{c}
\frac{1}{2}\|\bar{\mathbf{{x}}}-\mathbf{{x}}^{*}\|^{2}+\underset{\alpha\in\bar{\mathcal{E}}\cup\mathcal{V}}{\overset{}{\sum}}\mathbf{f}_{\alpha}(\mathbf{{x}}^{*})-F(\{\mathbf{{z}}_{\alpha}^{n,w}\}_{\alpha\in\bar{\mathcal{E}}\cup\mathcal{V}})\end{array}\\
 & \overset{\eqref{eq:From-8}}{\geq} & \begin{array}{c}
\frac{1}{2}\left\Vert \bar{\mathbf{{x}}}-\mathbf{{x}}^{*}-\underset{\alpha\in\bar{\mathcal{E}}\cup\mathcal{V}}{\overset{}{\sum}}\mathbf{{z}}_{\alpha}^{n,w}\right\Vert ^{2}\overset{\eqref{eq:from-10}}{=}\frac{1}{2}\|\bar{\mathbf{{x}}}-\mathbf{{x}}^{*}-\mathbf{{v}}_{A}^{n,w}\|^{2}.\end{array}
\end{eqnarray*}
The conclusion is immediate.

\textbf{Proof of (v): } We first make use of the technique in \cite[Lemma 29.1]{BauschkeCombettes11}
(which is in turn largely attributed to \cite{BD86}) to show that
\begin{equation}
\begin{array}{c}
\underset{n\to\infty}{\liminf}\left[\left(\underset{w=1}{\overset{\bar{w}}{\sum}}\|\mathbf{{v}}_{A}^{n,w}-\mathbf{{v}}_{A}^{n,w-1}\|\right)\sqrt{n}\right]=0.\end{array}\label{eq:root-n-dec}
\end{equation}
Seeking a contradiction, suppose instead that there is an $\epsilon>0$
and $\bar{n}>0$ such that if $n>\bar{n}$, then $\left(\sum_{w=1}^{\bar{w}}\|\mathbf{{v}}_{A}^{n,w}-\mathbf{{v}}_{A}^{n,w-1}\|\right)\sqrt{n}>\epsilon$.
By the Cauchy Schwarz inequality, we have $\begin{array}{c}
\frac{\epsilon^{2}}{n}<\left(\underset{w=1}{\overset{\bar{w}}{\sum}}\|\mathbf{{v}}_{A}^{n,w}-\mathbf{{v}}_{A}^{n,w-1}\|\right)^{2}\leq\bar{w}\underset{w=1}{\overset{\bar{w}}{\sum}}\|\mathbf{{v}}_{A}^{n,w}-\mathbf{{v}}_{A}^{n,w-1}\|^{2}.\end{array}$ This contradicts the earlier claim in (i) that $\sum_{n=1}^{\infty}\sum_{w=1}^{\bar{w}}\|\mathbf{{v}}_{A}^{n,w}-\mathbf{{v}}_{A}^{n,w-1}\|^{2}$
is finite. 

Through \eqref{eq:root-n-dec}, we find a sequence $\{n_{k}\}_{k=1}^{\infty}$
such that 
\begin{equation}
\begin{array}{c}
\lim_{k\to\infty}\left[\left(\underset{w=1}{\overset{\bar{w}}{\sum}}\|\mathbf{{v}}_{A}^{n_{k},w}-\mathbf{{v}}_{A}^{n_{k},w-1}\|\right)\sqrt{n_{k}}\right]=0.\end{array}\label{eq:subseq-sqrt-limit}
\end{equation}
Recalling the assumption \eqref{eq:sqrt-growth-sum-z}, we get 
\begin{equation}
\begin{array}{c}
\underset{k\to\infty}{\lim}\left[\left(\underset{w=1}{\overset{\bar{w}}{\sum}}\|\mathbf{{v}}_{A}^{n_{k},w}-\mathbf{{v}}_{A}^{n_{k},w-1}\|\right)\|\mathbf{{z}}_{\alpha}^{n_{k},\bar{w}}\|\right]\overset{\eqref{eq:sqrt-growth-sum-z},\eqref{eq:subseq-sqrt-limit}}{=}0\mbox{ for all }\alpha\in\bar{\mathcal{E}}\cup\mathcal{V}.\end{array}\label{eq:lim-sum-norm-z}
\end{equation}
 Moreover, 
\begin{eqnarray}
|\langle\mathbf{{v}}_{A}^{n_{k},\bar{w}}-\mathbf{{v}}_{A}^{n_{k},p(n_{k},\alpha)},\mathbf{{z}}_{\alpha}^{n_{k},\bar{w}}\rangle| & \leq & \begin{array}{c}
\|\mathbf{{v}}_{A}^{n_{k},\bar{w}}-\mathbf{{v}}_{A}^{n_{k},p(n_{k},\alpha)}\|\|\mathbf{{z}}_{\alpha}^{n_{k},\bar{w}}\|\end{array}\label{eq:inn-pdt-sum-norm}\\
 & \leq & \begin{array}{c}
\left(\underset{w=1}{\overset{\bar{w}}{\sum}}\|\mathbf{{v}}_{A}^{n_{k},w}-\mathbf{{v}}_{A}^{n_{k},w-1}\|\right)\|\mathbf{{z}}_{\alpha}^{n_{k},\bar{w}}\|.\end{array}\nonumber 
\end{eqnarray}
By (ii), there exists a further subsequence of $\{\mathbf{{v}}_{A}^{n_{k},\bar{w}}\}_{k=1}^{\infty}$
which converges to some $\mathbf{{v}}_{A}^{*}\in\mathbb{R}^{m}$.
Combining \eqref{eq:lim-sum-norm-z} and \eqref{eq:inn-pdt-sum-norm}
gives (v).

\textbf{Proof of (vi):}  From earlier results, we obtain 
\begin{eqnarray}
 &  & \begin{array}{c}
-\underset{\alpha\in\bar{\mathcal{E}}\cup\mathcal{V}}{\overset{}{\sum}}\mathbf{f}_{\alpha}(\bar{\mathbf{{z}}}-\mathbf{{v}}_{A}^{*})\end{array}\label{eq:biggest-formula}\\
 & \overset{\eqref{eq:From-8}}{\leq} & \begin{array}{c}
\frac{1}{2}\|\bar{\mathbf{{x}}}-(\bar{\mathbf{{x}}}-\mathbf{{v}}_{A}^{*})\|^{2}-F(\{\mathbf{{z}}_{\alpha}^{n_{k},\bar{w}}\}_{\alpha\in\bar{\mathcal{E}}\cup\mathcal{V}})\end{array}\nonumber \\
 & \overset{\eqref{eq:Dykstra-dual-defn},\eqref{eq:stagnant-indices}}{=} & \begin{array}{c}
\frac{1}{2}\|\mathbf{{v}}_{A}^{*}\|^{2}+\underset{\alpha\in\bar{\mathcal{E}}_{n_{k}}\cup\mathcal{V}}{\overset{}{\sum}}\mathbf{f}_{\alpha}^{*}(\mathbf{{z}}_{\alpha}^{n_{k},p(n_{k},\alpha)})\end{array}\nonumber \\
 &  & \begin{array}{c}
+\underset{((i,j),\bar{k})\notin\bar{\mathcal{E}}_{n_{k}}}{\overset{}{\sum}}\mathbf{f}_{((i,j),\bar{k})}^{*}(\mathbf{{z}}_{((i,j),\bar{k})}^{n_{k},\bar{w}})-\langle\bar{\mathbf{{x}}},\mathbf{{v}}_{A}^{n_{k},\bar{w}}\rangle+\frac{1}{2}\|\mathbf{{v}}_{A}^{n_{k},\bar{w}}\|^{2}\end{array}\nonumber \\
 & \overset{\scriptsize\eqref{eq:error-deriv},\eqref{eq:zero-indices}}{=} & \begin{array}{c}
\frac{1}{2}\|\mathbf{{v}}_{A}^{*}\|^{2}+\underset{\alpha\in\bar{\mathcal{E}}_{n_{k}}\cup\mathcal{V}}{\overset{}{\sum}}\langle\bar{\mathbf{{x}}}-\mathbf{{v}}_{A}^{n_{k},p(n_{k},\alpha)},\mathbf{{z}}_{\alpha}^{n_{k},p(n_{k},\alpha)}\rangle+\underset{i\in\mathcal{V}_{4}}{\overset{}{\sum}}E_{i,n_{k}}-\underset{i\in\mathcal{V}_{4}}{\overset{}{\sum}}D_{i,n_{k}}\end{array}\nonumber \\
 &  & \begin{array}{c}
-\underset{\alpha\in\bar{\mathcal{E}}_{n_{k}}\cup\mathcal{V}}{\overset{}{\sum}}\mathbf{f}_{\alpha}(\bar{\mathbf{{x}}}-\mathbf{{v}}_{A}^{n_{k},p(n_{k},\alpha)})-\langle\bar{\mathbf{{x}}},\mathbf{{v}}_{A}^{n_{k},\bar{w}}\rangle+\frac{1}{2}\|\mathbf{{v}}_{A}^{n_{k},\bar{w}}\|^{2}\end{array}\nonumber \\
 & \overset{\eqref{eq:stagnant-indices}}{=} & \begin{array}{c}
\frac{1}{2}\|\mathbf{{v}}_{A}^{*}\|^{2}-\underset{\alpha\in\bar{\mathcal{E}}_{n_{k}}\cup\mathcal{V}}{\overset{}{\sum}}\langle\mathbf{{v}}_{A}^{n_{k},p(n_{k},\alpha)}-\mathbf{{v}}_{A}^{n_{k},\bar{w}},\mathbf{{z}}_{\alpha}^{n_{k},\bar{w}}\rangle\end{array}\nonumber \\
 &  & \begin{array}{c}
-\underset{\alpha\in\bar{\mathcal{E}}_{n_{k}}\cup\mathcal{V}}{\overset{}{\sum}}\mathbf{f}_{\alpha}(\bar{\mathbf{{x}}}-\mathbf{{v}}_{A}^{n_{k},p(n_{k},\alpha)})-\langle\bar{\mathbf{{x}}},\mathbf{{v}}_{A}^{n_{k},\bar{w}}\rangle+\underset{i\in\mathcal{V}_{4}}{\overset{}{\sum}}E_{i,n_{k}}-\underset{i\in\mathcal{V}_{4}}{\overset{}{\sum}}D_{i,n_{k}}\end{array}\nonumber \\
 &  & \begin{array}{c}
+\left\langle \bar{\mathbf{{x}}}-\mathbf{{v}}_{A}^{n_{k},\bar{w}},\underset{\alpha\in\bar{\mathcal{E}}_{n_{k}}\cup\mathcal{V}}{\overset{}{\sum}}\mathbf{{z}}_{\alpha}^{n_{k},p(n_{k},\alpha)}\right\rangle +\frac{1}{2}\|\mathbf{{v}}_{A}^{n_{k},\bar{w}}\|^{2}\end{array}\nonumber \\
 & \overset{\eqref{eq:from-10},\eqref{eq:zero-indices}}{=} & \begin{array}{c}
\frac{1}{2}\|\mathbf{{v}}_{A}^{*}\|^{2}-\frac{1}{2}\|\mathbf{{v}}_{A}^{n_{k},\bar{w}}\|^{2}-\underset{\alpha\in\bar{\mathcal{E}}_{n_{k}}\cup\mathcal{V}}{\overset{}{\sum}}\langle\mathbf{{v}}_{A}^{n_{k},p(n_{k},\alpha)}-\mathbf{{v}}_{A}^{n_{k},\bar{w}},\mathbf{{z}}_{\alpha}^{n_{k},\bar{w}}\rangle\end{array}\nonumber \\
 &  & \begin{array}{c}
-\underset{\alpha\in\bar{\mathcal{E}}_{n_{k}}\cup\mathcal{V}}{\overset{}{\sum}}\mathbf{f}_{\alpha}(\bar{\mathbf{{x}}}-\mathbf{{v}}_{A}^{n_{k},p(n_{k},\alpha)})+\underset{i\in\mathcal{V}_{4}}{\overset{}{\sum}}E_{i,n_{k}}-\underset{i\in\mathcal{V}_{4}}{\overset{}{\sum}}D_{i,n_{k}}.\end{array}\nonumber 
\end{eqnarray}

Since $\lim_{k\to\infty}\mathbf{{v}}_{A}^{n_{k},\bar{w}}=\mathbf{{v}}_{A}^{*}$,
we have $\lim_{k\to\infty}\frac{1}{2}\|\mathbf{{v}}_{A}^{*}\|^{2}-\frac{1}{2}\|\mathbf{{v}}_{A}^{n_{k},\bar{w}}\|^{2}=0$.
The third term in the last group of formulas (i.e., the sum involving
the inner products) converges to 0 by (v). The term $\lim_{k\to\infty}\sum_{i\in\mathcal{V}_{4}}E_{i,n_{k}}$
equals to 0 by (iii). 

Next, recall that if $((i,j),\bar{k})\in\bar{\mathcal{E}}_{n_{k}}$,
by \eqref{eq:primal-subpblm}, we have $\bar{\mathbf{{x}}}-\mathbf{{v}}_{A}^{n_{k},p(n_{k},((i,j),\bar{k}))}\in H_{((i,j),\bar{k})}$.
Note that from Claim \ref{claim:Fenchel-duality}(b), we have $\bar{\mathbf{{x}}}-\mathbf{{v}}_{A}^{n,p(n,((i,j),\bar{k}))}\in H_{((i,j),\bar{k})}$
for all $((i,j),\bar{k})\in\bar{\mathcal{E}}_{n}$. There is a constant
$\kappa_{\bar{\mathcal{E}}_{n_{k}}}>0$ such that 
\begin{eqnarray}
 &  & d(\bar{\mathbf{{x}}}-\mathbf{{v}}_{A}^{n_{k},\bar{w}},\cap_{((i,j),\bar{k})\in\bar{\mathcal{E}}}H_{((i,j),\bar{k})})\label{eq:reg-argument}\\
 & \overset{\scriptsize{\bar{\mathcal{E}}_{n_{k}}\mbox{ connects }\mathcal{V},\mbox{ Prop \ref{prop:E-connects-V}(1)}}}{=} & d(\bar{\mathbf{{x}}}-\mathbf{{v}}_{A}^{n_{k},\bar{w}},\cap_{((i,j),\bar{k})\in\bar{\mathcal{E}}_{n_{k}}}H_{((i,j),\bar{k})})\nonumber \\
 & \leq & \kappa_{\bar{\mathcal{E}}_{n_{k}}}\max_{((i,j),\bar{k})\in\bar{\mathcal{E}}_{n_{k}}}d(\bar{\mathbf{{x}}}-\mathbf{{v}}_{A}^{n_{k},\bar{w}},H_{((i,j),\bar{k})})\nonumber \\
 & \overset{\bar{\mathbf{{x}}}-\mathbf{{v}}_{A}^{n_{k},p(n_{k},((i,j),\bar{k}))}\in H_{((i,j),\bar{k})}}{\leq} & \kappa_{\bar{\mathcal{E}}_{n_{k}}}\max_{((i,j),\bar{k})\in\bar{\mathcal{E}}_{n_{k}}}\|\mathbf{{v}}_{A}^{n_{k},\bar{w}}-\mathbf{{v}}_{A}^{n_{k},p(n_{k},((i,j),\bar{k}))}\|.\nonumber 
\end{eqnarray}
Let $\kappa:=\max\{\kappa_{\mathcal{\bar{E}}'}:\bar{\mathcal{E}}'\mbox{ connects }\mathcal{V}\}$.
We have $\kappa_{\bar{\mathcal{E}}_{n_{k}}}\leq\kappa$. Taking limits
of \eqref{eq:reg-argument}, the RHS converges to zero by (i), so
$d(\bar{\mathbf{{x}}}-\mathbf{{v}}_{A}^{*},\cap_{((i,j),\bar{k})\in\bar{\mathcal{E}}}H_{((i,j),\bar{k})})=0$,
or $\bar{\mathbf{{x}}}-\mathbf{{\mathbf{v}}}_{A}^{*}\in\cap_{((i,j),\bar{k})\in\bar{\mathcal{E}}}H_{((i,j),\bar{k})}$.
So $\sum_{((i,j),\bar{k})\in\bar{\mathcal{E}}}\mathbf{f}_{((i,j),\bar{k})}(\bar{\mathbf{{x}}}-\mathbf{{v}}_{A}^{*})=0$.
Together with the fact that $\bar{\mathbf{{x}}}-\mathbf{{v}}_{A}^{n_{k},p(n_{k},((i,j),\bar{k}))}\in H_{((i,j),\bar{k})}$,
we have 
\begin{equation}
\sum_{((i,j),\bar{k})\in\bar{\mathcal{E}}_{n_{k}}}\mathbf{f}_{((i,j),\bar{k})}(\bar{\mathbf{{x}}}-\mathbf{{v}}_{A}^{n_{k},p(n_{k},((i,j),\bar{k}))})=0=\underset{((i,j),\bar{k})\in\bar{\mathcal{E}}}{\overset{}{\sum}}\mathbf{f}_{((i,j),\bar{k})}(\bar{\mathbf{{x}}}-\mathbf{{v}}_{A}^{*}).\label{eq:all-indicator-edges-zero}
\end{equation}

Lastly, by the lower semicontinuity of $\mathbf{f}_{i}(\cdot)$, we
have 
\begin{equation}
\begin{array}{c}
-\underset{k\to\infty}{\lim}\underset{i\in\mathcal{V}}{\sum}\mathbf{f}_{i}(\bar{\mathbf{{x}}}-\mathbf{{v}}_{A}^{n_{k},p(n_{k},i)})\leq-\underset{i\in\mathcal{V}}{\overset{\phantom{\mathcal{V}}}{\sum}}\mathbf{f}_{i}(\bar{\mathbf{{x}}}-\mathbf{{v}}_{A}^{*}).\end{array}\label{eq:lsc-argument}
\end{equation}
As mentioned after \eqref{eq:biggest-formula}, taking the limits
as $k\to\infty$ would result in the first three terms and the 5th
term of the last formula in \eqref{eq:biggest-formula} to be zero.
Hence 
\begin{eqnarray*}
 &  & \begin{array}{c}
-\underset{\alpha\in\bar{\mathcal{E}}\cup\mathcal{V}}{\sum}\mathbf{f}_{\alpha}(\bar{\mathbf{{x}}}-\mathbf{{v}}_{A}^{*})\end{array}\\
 & \overset{\eqref{eq:biggest-formula}}{\leq} & \begin{array}{c}
\underset{k\to\infty}{\lim}-\underset{\alpha\in\bar{\mathcal{E}}_{n_{k}}\cup\mathcal{V}}{\sum}\mathbf{f}_{\alpha}(\bar{\mathbf{{x}}}-\mathbf{{v}}_{A}^{n_{k},p(n_{k},\alpha)})-\underset{k\to\infty}{\lim}\underset{i\in\mathcal{V}_{4}}{\sum}D_{i,n_{k}}\end{array}\\
 & \overset{\eqref{eq:all-indicator-edges-zero},\eqref{eq:lsc-argument},\eqref{eq:D-error-def}}{\leq} & \begin{array}{c}
-\underset{\alpha\in\mathcal{\bar{E}}\cup\mathcal{V}}{\sum}\mathbf{f}_{\alpha}(\bar{\mathbf{{x}}}-\mathbf{{v}}_{A}^{*}).\end{array}
\end{eqnarray*}
So \eqref{eq:biggest-formula} becomes an equation in the limit, and
$\lim_{n_{k}\to\infty}D_{i,n_{k}}=0$ for all $i\in\mathcal{V}_{4}$.
The first two lines of \eqref{eq:biggest-formula} then gives
\[
\begin{array}{c}
\underset{k\to\infty}{\lim}F(\{\mathbf{{z}}_{\alpha}^{n_{k},\bar{w}}\}_{\alpha\in\bar{\mathcal{E}}\cup\mathcal{V}})=\frac{1}{2}\|\mathbf{{v}}_{A}^{*}\|^{2}+\underset{i\in\mathcal{V}}{\sum}\mathbf{f}_{i}(\bar{\mathbf{{x}}}-\mathbf{{v}}_{A}^{*}),\end{array}
\]
which shows that $\bar{\mathbf{{x}}}-\mathbf{{v}}_{A}^{*}$ is the
primal minimizer. Recall the definitions of $F_{n,w}(\cdot)$, $F(\cdot)$
and $D_{i,n}$ in \eqref{eq:Dykstra-dual-defn}, \eqref{eq:Dykstra-dual-defn-1}
and \eqref{eq:D-error-def}. We recall \eqref{eq:stagnant-indices}.
Also, from line 11 of Algorithm \ref{alg:Ext-Dyk}, we have $\mathbf{f}_{\alpha,n,w}(\cdot)=\mathbf{f}_{\alpha,n,p(n,\alpha)}(\cdot)$.
This gives $F_{n_{k},\bar{w}}(\{\mathbf{{z}}_{\alpha}^{n_{k},\bar{w}}\}_{\alpha\in\bar{\mathcal{E}}\cup\mathcal{V}})+\sum_{i\in\mathcal{V}_{4}}D_{i,n_{k}}=F(\{\mathbf{{z}}_{\alpha}^{n_{k},\bar{w}}\}_{\alpha\in\bar{\mathcal{E}}\cup\mathcal{V}})$,
from which we deduce the equation on the left of \eqref{eq:thm-iv-concl}
as well. 
\end{proof}

\begin{proof}
[Proof of Proposition \ref{prop:control-growth}]Since this result
is used only in the proof of Theorem \ref{thm:convergence}(v), we
can make use of Theorem \ref{thm:convergence}(i) and (iii) in its
proof. To address condition (1), we can assume that $S_{n,w}\cap[\mathcal{V}_{1}\cup\mathcal{V}_{2}]$
always contains at most one element. Define the sets $\bar{S}_{n,1}$
and $\bar{S}_{n,2}\subset\{1,2,\dots\}\times\{1,\dots,\bar{w}\}$
as 
\begin{eqnarray*}
\bar{S}_{n,1} & := & \{(n',w):n'\leq n,\,w\in\{1,\dots,\bar{w}\},\,|S_{n',w}\cap\mathcal{V}|\leq1\}\\
\bar{S}_{n,2} & := & \{(n',w):n'\leq n,\,w\in\{1,\dots,\bar{w}\},\,|S_{n',w}\cap\mathcal{V}|>1\}.
\end{eqnarray*}
Either $S_{n',w}\cap[\mathcal{V}_{1}\cup\mathcal{V}_{2}]=\emptyset$
or $|S_{n',w}\cap[\mathcal{V}_{1}\cup\mathcal{V}_{2}]|=1$. In the
second case, let $i^{*}$ be the index such that $i^{*}\in S_{n',w}\cap[\mathcal{V}_{1}\cup\mathcal{V}_{2}]$.
Otherwise, in the first case, we let $i^{*}$ be any index in $[\mathcal{V}_{1}\cup\mathcal{V}_{2}]$.
We prove claims based on whether $(n',w)$ lies in $\bar{S}_{n,1}$
or $\bar{S}_{n,2}$. 

Without loss of generality, we can assume that $S_{n',w}\cap\bar{\mathcal{E}}$
are linearly independent constraints. This also means that for a $\mathbf{{v}}_{H}^{n',w}-\mathbf{{v}}_{H}^{n',w-1}$,
each $\mathbf{{z}}_{((i,j),\bar{k})}^{n',w}-\mathbf{{z}}_{((i,j),\bar{k})}^{n',w-1}$
can be determined uniquely with a linear map from the relation
\[
\sum_{\alpha\in\bar{\mathcal{E}}}[\mathbf{{z}}_{\alpha}^{n',w}-\mathbf{{z}}_{\alpha}^{n',w-1}]\overset{\eqref{eq:v-H-def}}{=}\mathbf{{v}}_{H}^{n',w}-\mathbf{{v}}_{H}^{n',w-1}.
\]
 Therefore for all $\alpha\in S_{n',w}\cap\bar{\mathcal{E}}$, there
is a constant $\kappa_{\alpha,S_{n',w}\cap\bar{\mathcal{E}}}>0$ such
that 
\begin{equation}
\|\mathbf{{z}}_{\alpha}^{n',w}-\mathbf{{z}}_{\alpha}^{n',w-1}\|\leq\kappa_{\alpha,S_{n',w}\cap\bar{\mathcal{E}}}\|\mathbf{{v}}_{H}^{n',w}-\mathbf{{v}}_{H}^{n',w-1}\|.\label{eq:basic-bdd-z-i-j}
\end{equation}
Thus there is a constant $\kappa>0$ such that 
\begin{equation}
\sum_{\alpha\in\bar{\mathcal{E}}}\|\mathbf{{z}}_{\alpha}^{n',w}-\mathbf{{z}}_{\alpha}^{n',w-1}\|\overset{\scriptsize{\mbox{Alg \ref{alg:Ext-Dyk} line 14}}}{=}\sum_{\alpha\in S_{n,w}\cap\bar{\mathcal{E}}}\|\mathbf{{z}}_{\alpha}^{n',w}-\mathbf{{z}}_{\alpha}^{n',w-1}\|\overset{\eqref{eq:basic-bdd-z-i-j}}{\leq}\kappa\|\mathbf{{v}}_{H}^{n',w}-\mathbf{{v}}_{H}^{n',w-1}\|.\label{eq:bdd-z-i-j}
\end{equation}

\textbf{\uline{Claim 1: If \mbox{$(n',w)\in\bar{S}_{n,1}$}, then
there is a constant \mbox{$C_{2}>1$} such that }}
\begin{equation}
\begin{array}{c}
\|\mathbf{{v}}_{H}^{n',w}-\mathbf{{v}}_{H}^{n',w-1}\|+\underset{\alpha\in\bar{\mathcal{E}}}{\sum}\|\mathbf{{z}}_{\alpha}^{n',w}-\mathbf{{z}}_{\alpha}^{n',w-1}\|+\underset{i\in\mathcal{V}}{\overset{\phantom{i\in\mathcal{V}}}{\sum}}\|\mathbf{{z}}_{i}^{n',w}-\mathbf{{z}}_{i}^{n',w-1}\|\leq C_{2}\|\mathbf{{v}}_{A}^{n',w}-\mathbf{{v}}_{A}^{n',w-1}\|.\end{array}\label{eq:all-3-bdd}
\end{equation}

We have 
\begin{eqnarray}
\begin{array}{c}
\underset{i\in\mathcal{V}}{\sum}[\mathbf{{v}}_{A}^{n',w}-\mathbf{{v}}_{A}^{n',w-1}]_{i}\end{array} & \overset{\scriptsize{\eqref{eq_m:all_acronyms}}}{=} & \begin{array}{c}
\underset{i\in\mathcal{V}}{\sum}\underset{\alpha\in S_{n',w}}{\sum}[\mathbf{{z}}_{\alpha}^{n',w}-\mathbf{{z}}_{\alpha}^{n',w-1}]_{i}\end{array}\nonumber \\
 & \overset{\mathbf{{z}}_{((i,j),\bar{k})}\in D^{\perp},\eqref{eq:D-and-D-perp}}{=} & \begin{array}{c}
\underset{i\in\mathcal{V}}{\sum}[\mathbf{{z}}_{i^{*}}^{n',w}-\mathbf{{z}}_{i^{*}}^{n',w-1}]_{i}\end{array}\nonumber \\
 & \overset{\scriptsize{\mbox{Prop \ref{prop:sparsity}}}}{=} & \begin{array}{c}
[\mathbf{{z}}_{i^{*}}^{n',w}-\mathbf{{z}}_{i^{*}}^{n',w-1}]_{i^{*}}.\end{array}\label{eq:for-norm-rate}
\end{eqnarray}
Recall that the norm $\|\cdot\|$ always refers to the $2$-norm unless
stated otherwise. By the equivalence of norms in finite dimensions,
we can find a constant $c_{1}$ such that 
\begin{eqnarray}
\begin{array}{c}
\|\mathbf{{v}}_{A}^{n',w}-\mathbf{{v}}_{A}^{n',w-1}\|\end{array} & \geq & \begin{array}{c}
c_{1}\underset{i\in\mathcal{V}}{\sum}\|[\mathbf{{v}}_{A}^{n',w}-\mathbf{{v}}_{A}^{n',w-1}]_{i}\|\end{array}\label{eq:bdd-z-i}\\
 & \geq & \begin{array}{c}
c_{1}\Big\|\underset{i\in\mathcal{V}}{\sum}[\mathbf{{v}}_{A}^{n',w}-\mathbf{{v}}_{A}^{n',w-1}]_{i}\Big\|\end{array}\nonumber \\
 & \overset{\eqref{eq:for-norm-rate}}{=} & \begin{array}{c}
c_{1}\|\mathbf{{z}}_{i^{*}}^{n,w}-\mathbf{{z}}_{i^{*}}^{n,w-1}\|\end{array}\nonumber \\
 & \overset{\scriptsize{\mbox{Alg \ref{alg:Ext-Dyk} line 14}}}{=} & \begin{array}{c}
c_{1}\underset{i\in\mathcal{V}}{\sum}\|\mathbf{{z}}_{i}^{n,w}-\mathbf{{z}}_{i}^{n,w-1}\|.\end{array}\nonumber 
\end{eqnarray}
Next, $\mathbf{{v}}_{H}^{n',w}-\mathbf{{v}}_{H}^{n',w-1}\overset{\eqref{eq:from-10}}{=}\mathbf{{v}}_{A}^{n',w}-\mathbf{{v}}_{A}^{n',w-1}-(\mathbf{{z}}_{i^{*}}^{n',w}-\mathbf{{z}}_{i^{*}}^{n',w-1})$,
so 
\begin{eqnarray}
\begin{array}{c}
\|\mathbf{{v}}_{H}^{n',w}-\mathbf{{v}}_{H}^{n',w-1}\|\end{array} & \leq & \begin{array}{c}
\|\mathbf{{v}}_{A}^{n',w}-\mathbf{{v}}_{A}^{n',w-1}\|+\|\mathbf{{z}}_{i^{*}}^{n',w}-\mathbf{{z}}_{i^{*}}^{n',w-1}\|\end{array}\label{eq:bdd-v-H}\\
 & \overset{\eqref{eq:bdd-z-i}}{\leq} & \begin{array}{c}
\left(1+\frac{1}{c_{1}}\right)\|\mathbf{{v}}_{A}^{n',w}-\mathbf{{v}}_{A}^{n',w-1}\|.\end{array}\nonumber 
\end{eqnarray}
We can choose $\{\mathbf{{z}}_{\alpha}^{n,w}\}_{\alpha\in\bar{\mathcal{E}}}$
such that 
\begin{equation}
\sum_{\alpha\in S_{n',w}\cap\bar{\mathcal{E}}}[\mathbf{{z}}_{\alpha}^{n',w}-\mathbf{{z}}_{\alpha}^{n',w-1}]\overset{\scriptsize{\mbox{Alg \ref{alg:Ext-Dyk} line 14}}}{=}\sum_{\alpha\in\bar{\mathcal{E}}}[\mathbf{{z}}_{\alpha}^{n',w}-\mathbf{{z}}_{\alpha}^{n',w-1}]\overset{\eqref{eq:v-H-def}}{=}\mathbf{{v}}_{H}^{n',w}-\mathbf{{v}}_{H}^{n',w-1}.\label{eq:decomp-v-H}
\end{equation}
Combining \eqref{eq:bdd-z-i}, \eqref{eq:bdd-v-H} and \eqref{eq:bdd-z-i-j}
together shows that there is a constant $C_{2}>1$ such that \eqref{eq:all-3-bdd}
holds.$\hfill\triangle$

\textbf{\uline{Claim 2. If \mbox{$(n',w)\in\bar{S}_{n,2}$}, then
there is a constant \mbox{$C_{5}>0$} such that}} 
\begin{equation}
\begin{array}{c}
\|\mathbf{{v}}_{H}^{n',w}-\mathbf{{v}}_{H}^{n',w-1}\|+\underset{\alpha\in\bar{\mathcal{E}}}{\sum}\|\mathbf{{z}}_{\alpha}^{n',w}-\mathbf{{z}}_{\alpha}^{n',w-1}\|+\underset{i\in\mathcal{V}}{\overset{\phantom{i\in\mathcal{V}}}{\sum}}\|\mathbf{{z}}_{i}^{n',w}-\mathbf{{z}}_{i}^{n',w-1}\|\leq C_{5}.\end{array}\label{eq:all-3-bdd-2}
\end{equation}
We have

\begin{eqnarray}
 &  & \begin{array}{c}
\|\mathbf{{v}}_{A}^{n',w}-\mathbf{{v}}_{A}^{n',w-1}\|\end{array}\label{eq:to-bdd-i-star-terms}\\
 & \geq & \begin{array}{c}
c_{1}\Big\|\underset{i\in\mathcal{V}}{\sum}[\mathbf{{v}}_{A}^{n',w}-\mathbf{{v}}_{A}^{n',w-1}]_{i}\Big\|\end{array}\nonumber \\
 & \overset{\eqref{eq_m:all_acronyms},\mathbf{{z}}_{((i,j),\bar{k})}\in D^{\perp},\eqref{eq:D-and-D-perp}}{=} & \begin{array}{c}
c_{1}\Big\|\underset{i\in\mathcal{V}}{\sum}\underset{j\in\mathcal{V}}{\sum}[\mathbf{{z}}_{j}^{n',w}-\mathbf{{z}}_{j}^{n',w-1}]_{i}\Big\|\end{array}\nonumber \\
 & \overset{\scriptsize{\mbox{Prop. }\ref{prop:sparsity}}}{=} & \begin{array}{c}
c_{1}\Big\|\underset{i\in\mathcal{V}}{\sum}[\mathbf{{z}}_{i}^{n',w}-\mathbf{{z}}_{i}^{n',w-1}]_{i}\Big\|\end{array}\nonumber \\
 & = & \begin{array}{c}
c_{1}\Big\|[\mathbf{{z}}_{i^{*}}^{n',w}-\mathbf{{z}}_{i^{*}}^{n',w-1}]_{i^{*}}+\underset{i\in\mathcal{V}_{3}\cup\mathcal{V}_{4}}{\sum}[\mathbf{{z}}_{i}^{n',w}-\mathbf{{z}}_{i}^{n',w-1}]_{i}\Big\|.\end{array}\nonumber 
\end{eqnarray}
From Theorem \ref{thm:convergence}(i), there is a constant $C_{3}>0$
such that $\|\mathbf{{v}}_{A}^{n',w}-\mathbf{{v}}_{A}^{n',w-1}\|\leq C_{3}$.
By Theorem \ref{thm:convergence}(iii), there is a constant $C_{4}>0$
such that 
\begin{equation}
\|\mathbf{{z}}_{i}^{n',w}-\mathbf{{z}}_{i}^{n',w}\|\leq C_{4}\mbox{ for all }i\in\mathcal{V}_{3}\cup\mathcal{V}_{4}.\label{eq:final-i-bdd}
\end{equation}
So $\|\mathbf{{z}}_{i^{*}}^{n',w}-\mathbf{{z}}_{i^{*}}^{n',w-1}\|=\|[\mathbf{{z}}_{i^{*}}^{n',w}-\mathbf{{z}}_{i^{*}}^{n',w-1}]_{i^{*}}\|\overset{\eqref{eq:to-bdd-i-star-terms},\eqref{eq:final-i-bdd}}{\leq}(|\mathcal{V}_{3}|+|\mathcal{V}_{4}|)C_{4}+\frac{1}{c_{1}}C_{3}$,
and 
\begin{equation}
\begin{array}{c}
\underset{i\in\mathcal{V}}{\sum}\|\mathbf{{z}}_{i}^{n',w}-\mathbf{{z}}_{i}^{n',w-1}\|\overset{\eqref{eq:to-bdd-i-star-terms},\eqref{eq:final-i-bdd}}{\leq}2(|\mathcal{V}_{3}|+|\mathcal{V}_{4}|)C_{4}+\frac{1}{c_{1}}C_{3}.\end{array}\label{eq:final-star-bdd}
\end{equation}
 Next, from \eqref{eq_m:all_acronyms}, we have 
\begin{eqnarray}
\begin{array}{c}
\mathbf{{v}}_{A}^{n',w}-\mathbf{{v}}_{A}^{n',w-1}\end{array} & = & \begin{array}{c}
\mathbf{{v}}_{H}^{n',w}-\mathbf{{v}}_{H}^{n',w-1}+[\mathbf{{z}}_{i^{*}}^{n',w}-\mathbf{{z}}_{i^{*}}^{n',w-1}]+\underset{i\in\mathcal{V}_{3}\cup\mathcal{V}_{4}}{\sum}[\mathbf{{z}}_{i}^{n',w}-\mathbf{{z}}_{i}^{n',w-1}]\end{array}\nonumber \\
\begin{array}{c}
\|\mathbf{{v}}_{H}^{n',w}-\mathbf{{v}}_{H}^{n',w-1}\|\end{array} & \leq & \begin{array}{c}
\|\mathbf{{v}}_{A}^{n',w}-\mathbf{{v}}_{A}^{n',w-1}\|+\|\mathbf{{z}}_{i^{*}}^{n',w}-\mathbf{{z}}_{i^{*}}^{n',w-1}\|+\underset{i\in\mathcal{V}_{3}\cup\mathcal{V}_{4}}{\sum}\|\mathbf{{z}}_{i}^{n',w}-\mathbf{{z}}_{i}^{n',w-1}\|\end{array}\nonumber \\
 & \leq & \begin{array}{c}
C_{3}+2(|\mathcal{V}_{3}|+|\mathcal{V}_{4}|)C_{4}+\frac{1}{c_{1}}C_{3}.\end{array}\label{eq:final-H-bdd}
\end{eqnarray}
Combining \eqref{eq:bdd-z-i-j}, \eqref{eq:final-star-bdd} and \eqref{eq:final-H-bdd},
we can show that Claim 2 holds. $\hfill\triangle$

Since $\{\mathbf{{z}}_{\alpha}^{n,0}\}_{\alpha\in\bar{\mathcal{E}}}$
was chosen to satisfy \eqref{eq_m:resetted-z-i-j}, there is some
$M>1$ such that 
\begin{equation}
\begin{array}{c}
\underset{\alpha\in\bar{\mathcal{E}}}{\sum}\|\mathbf{{z}}_{\alpha}^{n,0}\|\overset{\eqref{eq:reset-z-i-j-3}}{\leq}M\|\mathbf{{v}}_{H}^{n,0}\|\overset{\eqref{eq:reset-z-i-j-4}}{\leq}M\left(\|\mathbf{{v}}_{H}^{1,0}\|+\underset{n'=1}{\overset{n-1}{\sum}}\underset{w=1}{\overset{\bar{w}}{\sum}}\|\mathbf{{v}}_{H}^{n',w}-\mathbf{{v}}_{H}^{n',w-1}\|\right)\end{array}\label{eq:z-bdd-for-E}
\end{equation}
Now for any $n\geq1$, we have 
\begin{eqnarray}
\sum_{\alpha\in\bar{\mathcal{E}}\cup\mathcal{V}}\|\mathbf{{z}}_{\alpha}^{n,\bar{w}}\| & \leq & \sum_{\alpha\in\bar{\mathcal{E}}}\|\mathbf{{z}}_{\alpha}^{n,0}\|+\sum_{w=1}^{\bar{w}}\sum_{\alpha\in\bar{\mathcal{E}}}\|\mathbf{{z}}_{\alpha}^{n,w}-\mathbf{{z}}_{\alpha}^{n,w-1}\|\label{eq:2nd-big-ineq}\\
 &  & +\sum_{n'=1}^{n}\sum_{w=1}^{\bar{w}}\sum_{\alpha\in\mathcal{V}}\|\mathbf{{z}}_{\alpha}^{n',w}-\mathbf{{z}}_{\alpha}^{n',w-1}\|+\sum_{\alpha\in\mathcal{V}}\|\mathbf{{z}}_{\alpha}^{1,0}\|\nonumber \\
 & \overset{\eqref{eq:z-bdd-for-E}}{\leq} & M\|\mathbf{{v}}_{H}^{1,0}\|+\sum_{\alpha\in\mathcal{V}}\|\mathbf{{z}}_{\alpha}^{1,0}\|+\sum_{w=1}^{\bar{w}}\left(\sum_{\alpha\in\bar{\mathcal{E}}}\|\mathbf{{z}}_{\alpha}^{n',w}-\mathbf{{z}}_{\alpha}^{n',w-1}\|\right)\nonumber \\
 &  & +\sum_{n'=1}^{n-1}\sum_{w=1}^{\bar{w}}\left(M\|\mathbf{{v}}_{H}^{n',w}-\mathbf{{v}}_{H}^{n',w-1}\|+\sum_{\alpha\in\mathcal{V}}\|\mathbf{{z}}_{\alpha}^{n',w}-\mathbf{{z}}_{\alpha}^{n',w-1}\|\right)\nonumber \\
 & \overset{\eqref{eq:all-3-bdd},\eqref{eq:all-3-bdd-2}}{\leq} & M\|\mathbf{{v}}_{H}^{1,0}\|+\sum_{\alpha\in\mathcal{V}}\|\mathbf{{z}}_{\alpha}^{1,0}\|+MC_{2}\sum_{m=1}^{n}\sum_{w=1}^{\bar{w}}\|\mathbf{{v}}_{A}^{n',w}-\mathbf{{v}}_{A}^{n',w-1}\|\nonumber \\
 &  & +MC_{5}\left(M_{1}\sqrt{n}+M_{2}\right).\nonumber 
\end{eqnarray}
By the Cauchy Schwarz inequality, we have 
\begin{equation}
\sum_{n'=1}^{n}\sum_{w=1}^{\bar{w}}\|\mathbf{{v}}_{A}^{n',w}-\mathbf{{v}}_{A}^{n',w-1}\|\leq\sqrt{n\bar{w}}\sqrt{\sum_{n'=1}^{n}\sum_{w=1}^{\bar{w}}\|\mathbf{{v}}_{A}^{n',w}-\mathbf{{v}}_{A}^{n',w-1}\|^{2}}.\label{eq:sum-bdd-by-sqrt-n}
\end{equation}
Since the second square root of the right hand side of \eqref{eq:sum-bdd-by-sqrt-n}
is bounded by Theorem \ref{thm:convergence}(i), we make use of \eqref{eq:2nd-big-ineq}
to obtain the conclusion \eqref{eq:sqrt-growth-sum-z} as needed. 
\end{proof}
\bibliographystyle{amsalpha}
\bibliography{../refs}

\end{document}